\documentclass[11pt]{amsart}
\usepackage{latexsym}
\usepackage{amsmath}
\usepackage{amsfonts}
\usepackage{amssymb}
\usepackage{amsthm}
\usepackage{bm}

\usepackage{enumitem}

\renewcommand{\b}{{\bf b}}
\renewcommand{\d}{{\bf d}}
\newcommand{\h}{{\bf h}}
 \renewcommand{\v}{{\bf v}}
  \newcommand{\n}{{\bf n}}
  \newcommand{\m}{{\bf m}}
  \newcommand{\e}{\bm{\epsilon}}
  \newcommand{\LE}{\mathcal{LE}}
   \renewcommand{\H}{\mathcal{H}}
  \newcommand{\B}{\mathcal{B}}
\newcommand{\A}{\mathcal{A}}
 \newcommand{\cB}{\mathcal{B}}
 \newcommand{\cA}{\mathcal{A}}
 \newcommand{\cC}{\mathcal{C}}

\newcommand{\cD}{\mathcal{D}}

\renewcommand{\P}{\mathcal{A}}
\newcommand{\G}{\mathcal{G}}
\newcommand{\cG}{\mathcal{G}}
\newcommand{\X}{\mathcal{X}}
\newcommand{\F}{\mathcal{F}}
\newcommand{\cF}{\mathcal{F}}
\newcommand{\cI}{\mathcal{I}}

\newcommand{\CZ}{\mathcal{Z}}

\newcommand{\cZ}{\mathcal{Z}}
\newcommand{\gG}{\Gamma}
\newcommand{\C}{\mathbb{C}}

\newcommand{\Q}{\mathcal{B}}
\newcommand{\R}{\mathbb{R}}
\newcommand{\E}{\mathbb{E}}
\newcommand{\N}{\mathbb{N}}

\newcommand{\cN}{\mathcal{N}}

\newcommand{\Z}{\mathbb{Z}}

\renewcommand{\=}{\mathrel{\mathop:}=}

\DeclareMathOperator{\vdc}{-vdC}

\newcommand{\norm}[1]{\left\Vert #1\right\Vert}
\newcommand{\nnorm}[1]{\lvert\!|\!| #1|\!|\!\rvert}
\theoremstyle{plain}
\newtheorem{theorem}{Theorem}[section]
\newtheorem{lemma}[theorem]{Lemma}
\newtheorem{proposition}[theorem]{Proposition}

\newtheorem*{claim*}{Claim}
\newtheorem*{lemma*}{Lemma}

\newtheorem*{theoremA'}{Theorem A'}
\newtheorem*{theoremB'}{Theorem B'}
\newtheorem*{theoremC'}{Theorem C'}

\newtheorem{problem}{Problem}
\newtheorem*{theorem*}{Theorem}

\newtheorem{corollary}[theorem]{Corollary}

\theoremstyle{definition}
\newtheorem{definition}[theorem]{Definition}

\newtheorem*{example*}{Example}

\newtheorem*{definition*}{Definition}

\theoremstyle{remark}

\newtheorem*{remark}{Remark}

\begin{document}

\title{A multidimensional Szemer\'edi theorem  for Hardy sequences of
different growth }
%%Ergodic averages of commuting transformations and
%%Hardy sequences of different growth
\author{Nikos Frantzikinakis}
\address[Nikos Frantzikinakis]{University of Crete, Department of mathematics, Knossos Avenue, Heraklion 71409, Greece} \email{frantzikinakis@gmail.com}

\begin{abstract}
We prove a variant of the multidimensional polynomial Szemer\'edi
theorem of Bergelson and Leibman where one replaces polynomial
sequences with other sparse sequences defined by functions that
belong to some Hardy field and satisfy certain growth conditions. We
do this by studying the limiting behavior of the corresponding
multiple ergodic averages and obtaining a simple limit formula. A
consequence of this formula in topological dynamics shows denseness
of  certain orbits when the iterates are restricted to suitably
chosen sparse subsequences. Another consequence is that every
syndetic set of integers contains certain non-shift invariant
patterns, and   every finite coloring of $\N$, with each color class
a syndetic set, contains certain polychromatic patterns, results
very particular to our non-polynomial setup.

%%The tools used involve ergodic decomposition
%%results,
%% elementary estimates involving uniformity seminorms,
%%and equidistribution results of sequences on nilmanifolds.
\end{abstract}

\thanks{The author was partially supported by  Marie Curie IRG  248008.}

\subjclass[2000]{Primary: 37A45;  Secondary:  28D05, 05D10, 11B25}

\keywords{Ergodic theory, recurrence,   Hardy field, Ramsey theory,
nilmanifolds.}

\maketitle
%%\setcounter{tocdepth}{1}%%Only sections appear in Contents
%%\tableofcontents

\section{Introduction}
In \cite{Fu77},  Furstenberg  gave an ergodic theoretic proof of
Szemer\'edi's theorem on arithmetic progressions, and  using similar
methods, Furstenberg and Katznelson \cite{FuK79} proved a
multidimensional extension of Szemer\'edi's theorem. Later on,
Bergelson and Leibman \cite{BL96} gave a polynomial extension of
this result, a special case of which states  that given any
collection of polynomials $p_1,\ldots, p_\ell\colon \N\to\Z$, with
zero constant term, and vectors $\v_1,\ldots, \v_\ell \in \Z^d$,
every subset of $\Z^d$  of positive upper  density contains
configurations of the form
\begin{equation}\label{E:polynomial1}
\{\v, \ \v+p_1(n)\v_1, \ \ldots\ ,\v+p_\ell(n)\v_\ell\}
\end{equation}
for some $\v\in \Z^d$ and $n\in \N$. In the course of proving this
result they introduced and studied the limiting behavior in
$L^2(\mu)$ of the following multiple ergodic averages
\begin{equation}\label{E:polynomial2}
\frac{1}{N}\sum_{n=1}^N f_1(T_1^{p_1(n)}x)\ \! \cdots\ \!
f_\ell(T_\ell^{p_\ell(n)}x),
\end{equation}
 where $T_1,\ldots,
T_\ell\colon X\to X$ are invertible commuting measure preserving
transformations acting on some  probability space $(X,\X,\mu)$ and
$f_1,\ldots, f_\ell\in L^\infty(\mu)$. Their goal was to prove a
multiple recurrence property, namely, that for every $A\in \X$ with
$\mu(A)>0$ one has
\begin{equation}\label{E:polynomial3}
\liminf_{N\to\infty} \frac{1}{N}\sum_{n=1}^N \mu(A\cap
T_1^{-p_1(n)}A\cap \cdots \cap T_\ell^{-p_\ell(n)}A)>0.
\end{equation}
From this, the combinatorial result follows via the correspondence
principle of Furstenberg \cite{Fu77, Fu81a}. Bergelson and Leibman
managed to prove this multiple recurrence property without getting
very precise information about the limit of the averages
\eqref{E:polynomial2}. Important role in their proof played an
ergodic structure theorem (already present in \cite{FuK79}) and the
coloristic counterpart of their density result, now known as
polynomial van der Waerden theorem, which they proved using more
elementary methods.\footnote{When $T_1=\cdots=T_\ell$,  using deep
results  from
 \cite{HK05a, HK05b, Lei05a, Zi07}, property
\eqref{E:polynomial3} was proved in \cite{BLL08} without appealing
to the polynomial van der Waerden theorem. No such proof   for
general commuting transformations is   known.} The reader can find
several other examples were ergodic methods were used to prove
combinatorial results in the surveys \cite{Be06a, Be06b, Kra06b,
Kra11}.

%% Important role in the
%%proof of \eqref{E:polynomial3} plays an ergodic structure theorem
%%from \cite{FuK79} and a coloristic counterpart of the polynomial
%%Szemer\'edi theorem, now known as polynomial van der Waerden
%%theorem, which Bergelson and Leibman proved in \cite{BL96} using
%%more elementary methods. This last combinatorial input was key; it
%%enabled them  to sidestep the much harder problem of getting an in
%%depth understanding of  the limiting behavior  of the averages
%%\eqref{E:polynomial2}. In fact, mean convergence  of the averages
%%eqref{E:polynomial2} was only recently established in \cite{W11}
%%after a long series of important partial results that includes
%%\cite{CL88a, FuW96, HK05a, HK05b, Lei05c, Ta08}. Furthermore,  a
%%detailed description of the limit that allows to  prove
%%\eqref{E:polynomial3} without using the polynomial van der Waerden
%%theorem is only available in special cases \cite{BLL08, Chu11,
%%CFH11}.

In the present article, we  establish a variant of the
 polynomial Szemer\'edi theorem  where  one
replaces the polynomials $p_1,\ldots, p_\ell$ with a collection of
sparse sequences of integers defined using functions that belong to
some Hardy field and satisfy certain growth conditions. For
instance, we show that one can substitute the configurations
\eqref{E:polynomial1} with configurations of the form
$$
\{\v, \ \v+[n^{c_1}]\v_1,  \ \ldots \ ,\v+[n^{c_\ell}]\v_\ell\}
$$
for every choice of  distinct positive   non-integers $c_1,\ldots,
c_\ell$. Despite the similarity of this result with the polynomial
Szemer\'edi theorem,  its proof is very different. This is mainly
because we are unable to prove the corresponding coloristic result
in a simple way (the only proof we know uses the density result). To
circumvent this problem, we deviate from the classical methods used
in \cite{BL96,  FuK79}, and aim at proving the needed multiple
recurrence property by obtaining a complete understanding of
 the limiting behavior  of the corresponding multiple ergodic
 averages.
 %%This approach to proving multiple recurrence results has been
 %%utilized a lot in recent years,
  %% see for example  \cite{BHK05,BLL08,CFH11,Fr04, Fr08, Fr10, FrK06, FrLW09, FrW09,Gr11, Zi06, WZ11}.
  In our particular setup, we   establish the following
 explicit limit formula
\begin{equation}\label{E:fracpower1}
\lim_{N\to\infty}\frac{1}{N}\sum_{n=1}^N f_1(T_1^{[n^{c_1}]}x)\ \!
\cdots\ \! f_\ell(T_\ell^{[n^{c_\ell}]}x)= \tilde{f}_1(x) \cdots
\tilde{f}_\ell(x),
\end{equation}
%%$$
%%\lim_{N\to\infty}\frac{1}{N}\sum_{n=1}^N T_1^{[n^{c_1}]}f_1 \ \!
%%\cdots\ \! T_\ell^{[n^{c_\ell}]}f_\ell= \tilde{f}_1 \cdots
%%\tilde{f}_\ell
%%$$
where $c_1,\ldots, c_\ell$ are distinct positive  non-integers, the
convergence takes place in $L^2(\mu)$, and $\tilde{f}_i$ is the
orthogonal projection of the function $f_i$ on the subspace of
functions that are left invariant by the transformation $T_i$. The
proof of identity \eqref{E:fracpower1} relies on ergodic
decomposition results, seminorm  estimates, and equidistribution
results  on nilmanifolds.
%% An interesting interpretation of the limit formula
%%\eqref{E:fracpower1} in topological dynamics is given in
%%Section~\ref{SS:TopDyn}.

 Because of  the  explicit
evaluation  of the limit in  \eqref{E:fracpower1}, it is a simple
matter to prove a multiple recurrence property  analogous to
\eqref{E:polynomial3}, with an explicit lower bound, namely,
\begin{equation}\label{E:fracpower2}
\lim_{N\to\infty} \frac{1}{N}\sum_{n=1}^N \mu(A\cap
T_1^{-[n^{c_1}]}A\cap \cdots \cap T_\ell^{-[n^{c_\ell}]}A)\geq
(\mu(A))^{\ell+1}
\end{equation}
where, as usual,   $c_1,\ldots, c_\ell$ are distinct positive
non-integers.

 We remark that    identity \eqref{E:fracpower1} and  estimate
\eqref{E:fracpower2} fail if one of the numbers $c_1,\ldots, c_\ell$
is an integer greater than $1$. This is a known feature of
polynomial sequences  caused by their lack of equidistribution in
congruence classes.
%%Although a lot of effort has gone in studying
%%the limiting behavior of the polynomial averages
%%\eqref{E:polynomial2} (see \cite{Au09, Au11a, Au11b, Be87a, CFH11,
%%CL88a, Fr08, FrK06, Ho09, HK05a, HK05b, Joh11, Lei05c, Lei07, Lei09,
%%Ta08, W11}), the explicit form of their limit is not known.
In this
respect, fractional powers, as well as  other sequences that we
consider next, are better suited for the problems we are interested
in.

The method of proof of \eqref{E:fracpower1} %%is flexible enough to
allows us to work in a much more general setup. We prove that the
place of the sequences $[n^{c_1}],\ldots, [n^{c_\ell}]$ can take any
collection of sequences $[a_1(n)],\ldots, [a_\ell(n)]$, where the
functions $a_1(t), \ldots, a_\ell(t)$ belong to some Hardy field,
have different growth rates, and, roughly speaking, grow like a
fractional power of $t$ (for the exact statements see
Theorems~\ref{T:MainConv} and \ref{T:MainRec}). For instance,
%%an immediate consequence of  our  results is that one
we can
 use the
following collection of sequences
$$
\big\{[n^{c}(\log{n})^{d_1}],  \ \ldots  \
,[n^{c}(\log{n})^{d_\ell}]\big\}
$$
where  $c$ is  a positive non-integer and  $d_1,\ldots, d_\ell\in
\R$ are distinct, or more exotic collections like
$$
\Big\{[\sqrt[3]{n}],\ [n\sqrt{n^3+1}],\ [n^{3/2}
e^{\sqrt{\log\log{n}}}],\ [n^{\pi}/\log{n}],\  \Big[\sqrt{n}
\int_0^n e^{\sqrt{\log{t}}} \ dt\Big]\Big\}.
$$

Another interesting consequence of the limit formula
\eqref{E:fracpower1} is in topological dynamics. It  enables us to
show, for instance, that if $T,S$ are commuting minimal
transformations acting on a compact metric space $(X,d)$, and $a,b$
are distinct positive  non-integers, then for a residual set of
$x\in X$ one has
$$
\overline{\big((T^{[n^a]}x, S^{ [n^b]}x)\big)_{n\in\N}}=X\times
  X.
  $$
Periodic systems show that this fails if either $a$ or $b$ is an
integer greater than $1$.

 The limit formula
\eqref{E:fracpower1} also has  some rather unusual consequences in
combinatorics. It implies that if $E\subset \N$ is  syndetic (i.e.
finitely many translates of $E$  cover $\N$), then it contains
certain non-shift invariant patterns, for instance, we prove that
for  $a,b$ as before,  the system
\begin{align*}
 2y-x=&\ [n^a]\\
 3z-x=&\ [n^b]
\end{align*}
has a solution with $x,y,z\in E$ and $n\in\N$. It also implies that
for every finite coloring of $\N$, where each color class is a
syndetic set, the system
\begin{align*}
 y-x=&\ [n^a]\\
 z-x=&\ [n^b]
\end{align*}
has a solution with $x,y,z$ having arbitrary colors. Again, these
results are very particular to our non-polynomial setup and fail  if
either $a$ or $b$ is an integer greater than $1$.

 In the next
section we give a  precise formulation of our main results.
%% and also
%%introduce  some of the concepts used throughout the article.
\section{Main results}

\subsection{Our setup}
In order to properly state our results  we have to first
 introduce some notation.

A \emph{system} $(X,\X,\mu, T_1,\ldots, T_\ell)$ is a Lebesgue
probability space $(X,\X,\mu)$ together with  a collection of
commuting invertible measure preserving transformations $T_1,\ldots,
T_\ell\colon X\to X$. By  $\E(f|\cI_{T_i})$ we denote the
conditional expectation on the $\sigma$-algebra $\cI_{T_i}$ of
$T_i$-invariant sets. Equivalently, this is the orthogonal
projection on the closed subspace of $T_i$-invariant functions.

Throughout the article we use the symbol $\mathcal{H}$ to denote a
translation invariant Hardy field (all notions defined in
Section~\ref{SS:Hardy}). All iterates of the transformations
involved in our statements are  defined using functions that belong
to the same  Hardy field. This particular setup
  enables us to work within a rich class of functions and offers
several aesthetic and  technical advantages.

In most instances, we  restrict our attention to the following
``good'' class of functions:
\begin{definition}\label{D:good}
We denote by $\G$ the collection  of all functions $a\colon
[c,\infty)\to \R$ that satisfy the growth conditions $|a(t)|/(t^d
\log{t})\to \infty$ and $|a(t)|/t^{d+1}\to 0$ as $t\to\infty$ for
some integer $d\geq 0$.
\end{definition}
%%\begin{definition}
%%We define by $\cG$ to be the set of functions $a$ that satisfy the
%%growth conditions $|a(t)|/t^{d+\varepsilon}\to+\infty$ and
%%$t^{d+1}/|a(t)|\to+\infty$ as $t\to+\infty$
%%   for some integer $d\geq
%%0$ and $\varepsilon>0$.
%%\end{definition}
The presence of the logarithm on the first condition is purely for
technical reasons, it ensures that successive differences of
functions in $\cG\cap \H$ either converge to $0$ or else are
functions with substantial growth (this follows from
Lemma~\ref{L:properties}). The key features of functions in $\cG$
are: $(i)$ they do not grow very fast, and $(ii)$ they ``stay away''
from all polynomials in a rather strong sense. Staying away from
polynomials is a  property that we desire since the conclusions of
our main results    fail for some
 polynomials with integer coefficients.

\subsection{Results in ergodic theory}
For the sake of brevity  we define:
\begin{definition}
   The functions
$a_1,\ldots,a_\ell\colon [c,\infty)\to \R$  are said to have
\emph{different growth rates} if their pairwise  quotients converge
to $\pm \infty$ or to $0$.
\end{definition}
\subsubsection{The limit formula}
The main result of this article is the following limit formula (a
special case of this was stated as Problem 6 in \cite{Fr10} and as
Problem 29 in \cite{Fr11}):
\begin{theorem}\label{T:MainConv}
Let $\H$ be a Hardy field and  $a_1,\ldots,a_\ell \in \G\cap \H$ be
functions with different growth rates.
%%and satisfy $t^{k+\varepsilon} \prec a_i(t)\prec t^{k+1}$ for some $k=k_i\in \N$ and $\varepsilon>0$ .
Then  for every   system $(X,\mathcal{X},\mu,T_1,\ldots, T_\ell)$
 and functions
   $f_1,\dots,f_\ell\in L^\infty(\mu)$ we have
   \begin{equation}\label{E:ProductForm}
\lim_{N\to\infty}    \frac1N \sum_{n=1}^N T_1^{[a_1(n)]}f_1 \ \!
\cdots \ \!
    T_\ell^{ [a_\ell(n)]}f_\ell= \tilde{f}_1 \cdots
    \tilde{f}_\ell
  \end{equation}
  where $\tilde{f}_i\=\E(f_i|\cI_{T_i})=\lim_{N\to\infty}\frac{1}{N}\sum_{n=1}^NT_i^nf_i$ and the convergence takes
place in $L^2(\mu)$.
\end{theorem}
The case $\ell=1$ follows from the equidistribution results in
\cite{Bos94} and the case where all the iterates have sub-linear
growth follows form \cite{Fr10} (this case is simple and no
commutativity of the transformations is needed). When all the
transformations are equal a slightly weaker result is proved in
\cite{Fr10}.\footnote{Even in the case where all the transformations
are equal, our present argument   has a technical advantage over the
argument used in \cite{Fr10}. This enables us to relax the growth
condition used there.}
%%It does not come for free though, we have to use  a significantly
%%more complicated inductive scheme that preserves nice families of
%%pairs of transformations}
 Easy examples of
rational rotations on the circle show that for $\ell\geq 2$ the
limit formula \eqref{E:ProductForm} fails when  the iterates are
given by polynomial sequences, even if these polynomials have
distinct degrees. In fact,  it fails if some non-trivial linear
combination of the functions $a_1,\ldots, a_\ell$ is a polynomial
different than $ \pm t+c$.  When the assumption that the
transformations commute is removed, and two or more iterates have
super-linear growth, examples from \cite{FrLW11} show that the limit
in \eqref{E:ProductForm} does not in general exist. Lastly, we
remark that in \eqref{E:ProductForm} the limit
 $\lim_{N\to\infty}\frac{1}{N} \sum_{n=1}^N$ cannot be
replaced by the  uniform limit $\lim_{N-M\to\infty}\frac{1}{N-M}
\sum_{n=M}^N$. This is  because for  $a\in \H\cap \G$
 one can show that  the sequence $([a(n)])$ takes odd
 (respectively even) values in arbitrarily long intervals of integers.

\subsubsection{Multiple recurrence}
 Using Theorem~\ref{T:MainConv} we easily
deduce the following:
\begin{theorem}\label{T:MainRec}
Under the assumptions of Theorem~\ref{T:MainConv}, if $A_0, A_1,
\ldots, A_\ell\in \X$ satisfy
$$
 \mu(A_0\cap
T_1^{k_1}A_1\cap \cdots \cap T_\ell^{k_\ell}A_\ell)=\alpha>0
$$ for some
$k_1,\ldots, k_\ell\in \Z$, then
\begin{equation}\label{E:cvb}
\lim_{N\to\infty} \frac{1}{N}\sum_{n=1}^N\mu(A_0\cap
T_1^{-[a_1(n)]}A_1 \cap\cdots \cap T_\ell^{-[a_\ell(n)]}A_\ell)\geq
\alpha^{\ell +1}.
\end{equation}
\end{theorem}
\begin{proof}
By Theorem~\ref{T:MainConv}  it suffices to show that
$$
\int {\bf 1}_{A_0} \cdot \E({\bf 1}_{A_1}|\cI_{T_1})\cdots \E({\bf
1}_{A_\ell}|\cI_{T_\ell})\ d\mu\geq a^{\ell+1}.
$$
 Since each  function $\E({\bf 1}_{A_i}|\cI_{T_i})$ is
 $T_i$-invariant,    the left hand side is greater than
 $$
\int f\cdot \E(f|\cI_{T_1})\cdots \E(f|\cI_{T_\ell})\ d\mu\geq
\Big(\int f \ d\mu\Big)^{\ell+1}=a^{\ell+1},
$$
where
  $f={\bf 1}_{A_0\cap
T_1^{k_1}A_1\cap \cdots \cap T_\ell^{k_\ell}A_\ell}$ and the last
estimate follows from  Lemma 1.6 in \cite{Chu11}.
\end{proof}
Hence, the limit in \eqref{E:cvb}  is positive if $\mu(A_0)>0$ and
$\mu(\bigcup_{k\in \Z}T_i^{k}A_i)=1$ for $i=1,\ldots, \ell$.

Applying Theorem~\ref{T:MainRec}  for $A_0=\cdots=A_\ell=A$ and
$k_1=\cdots=k_\ell=0$ we deduce:
\begin{corollary}\label{C:MainRec}
Under the assumptions of Theorem~\ref{T:MainConv}, for every set
$A\in \X$ we have
\begin{equation}\label{E:cvb'}
\lim_{N\to\infty} \frac{1}{N}\sum_{n=1}^N\mu(A\cap T_1^{-[a_1(n)]}A
\cap\cdots \cap T_\ell^{-[a_\ell(n)]}A)\geq (\mu(A))^{\ell +1}.
\end{equation}
\end{corollary}
Comments similar to those made after the statement of
Theorem~\ref{T:MainConv} apply here too. Furthermore, if $\ell=2$
and $a_1=a_2$, then no power of $\mu(A)$ can be used as a lower
bound in \eqref{E:cvb'} (see Theorem~ 2.1 in \cite{BHK05}).

\subsection{Results in topological dynamics and combinatorics} \label{SS:TopDyn}
Let $(X,d)$ be  a compact metric space and  $T_1,\ldots,
T_\ell\colon X\to X$ be invertible commuting continuous
transformations. There exists a Borel measure that is left invariant
by all transformations. If in addition every transformation is
minimal (i.e. $\overline{(T_i^nx)_{n\in\N}}=X$ for every $x\in X$),
then this measure gives positive value to every non-empty open set,
and for every $x\in X$ and non-empty open set $U$ the set $\{n\in
\N\colon T_i^nx\in U\}$ has bounded gaps (see for example
\cite{Fu81a}). As a consequence, for every $x\in X$ and non-empty
open set $U$ we have $ \lim_{N\to\infty}\frac{1}{N}\sum_{n=1}^N{\bf
1}_U(T_i^nx)>0$, and using   Theorem~\ref{T:MainConv}  we get for
almost every $x\in X$ (and hence for a dense set of $x\in X$) that
$$
\limsup_{N\to\infty} \frac{1}{N}\sum_{n=1}^N {\bf
1}_{U_1}(T_1^{[a_1(n)]}x)\cdots {\bf
1}_{U_\ell}(T_\ell^{[a_\ell(n)]}x)>0
$$
whenever the sets   $U_1,\ldots, U_\ell$  are taken from a given
countable basis of non-open sets. Using this, we deduce the
following (the set of $x\in X$ for which \eqref{E:dense} holds is
trivially $G_\delta$ and $T_i$-invariant):
\begin{theorem}\label{T:TopDyn}
Let $\H$ be a Hardy field and  $a_1,\ldots,a_\ell \in \G\cap \H$ be
functions with different growth rates. Let $(X,d)$ be a compact
metric space and $T_1,\ldots, T_\ell\colon X\to X$ be invertible
commuting minimal transformations.
%%and satisfy $t^{k+\varepsilon} \prec a_i(t)\prec t^{k+1}$ for some $k=k_i\in \N$ and $\varepsilon>0$ .
Then for a residual and $T_i$-invariant set of  $x\in X$ we have
\begin{equation}\label{E:dense}
\overline{\big\{(T_1^{[a_1(n)]}x,\ldots, T_\ell^{
[a_\ell(n)]}x)\colon n\in \N \big\}}=X\times \cdots \times X.
\end{equation}
%%holds for a residual and $T_i$-invariant set of  $x\in X$.
\end{theorem}
 Examples in \cite{Pav08} show that even
when $\ell=1$ identity \eqref{E:dense} may fail for an uncountable
set of $x\in X$. In fact,
  for every sequence of integers $(a(n))$
with zero density, it is shown in \cite{Pav08} that there exists a
totally minimal and uniquely ergodic topological dynamical system
$(X,d,T)$ such that for an uncountable set of $x\in X$ one has
$x\notin \overline{\{T^{a(n)}x,n\in \N\}}$. Examples of minimal
rotations on finite cyclic groups show that if $p\in \Z[t]$ is any
polynomial $\neq \pm t +c$, then one may have
$\overline{\{T^{p(n)}x,n\in \N\}}\neq X$ for every $x\in X$.

Every continuous transformation $T$ on a compact metric space
$(X,d)$ has a non-empty closed $T$-invariant set $Y\subset X$ such
that the transformation $T\colon Y\to Y$ is minimal (see for example
\cite{Fu81a}).  Using this, and Theorem~\ref{T:TopDyn} for
$T_1=\cdots =T_\ell=T$, we deduce:
\begin{corollary}\label{C:TopDyn}
Let $\H$ be a Hardy field and  $a_1,\ldots,a_\ell \in \G\cap \H$ be
functions with different growth rates. Let $(X,d)$ be a compact
metric space and $T\colon X\to X$ be an invertible  continuous
transformation.
%%and satisfy $t^{k+\varepsilon} \prec a_i(t)\prec t^{k+1}$ for some $k=k_i\in \N$ and $\varepsilon>0$ .
Then for  a non-empty and    $T$-invariant set  of $x\in X$ we have
\begin{equation}\label{E:dense'}
\overline{\big\{(T^{[a_1(n)]}x,\ldots, T^{ [a_\ell(n)]}x)\colon n\in
\N \big\}}=\overline{\{T^nx\colon n\in\N\}}\times\cdots\times
\overline{\{T^nx\colon n\in\N\}} .
 \end{equation}
%% holds for  a non-empty and    $T$-invariant set  of $x\in X$.
\end{corollary}
%%The set for which \eqref{E:dense'} holds may be a singleton (take
%%$T\colon [0,1]\to [0,1]$   defined by $Tx=x/2$).
Again, simple examples show this result fails if $\ell=1$ and $p\in
\Z[t]$ is any polynomial $\neq \pm t +c$.

\subsection{Combinatorial consequences}
For a set $\Lambda\subset\mathbb{\Z}^d$, we define its upper density
by
$\bar{d}(\Lambda)\=\limsup_{N\to\infty}|\Lambda\cap[-N,N]^d|/(2N)^d$
(any other shift invariant mean   works for our purposes). Combining
the previous multiple recurrence result with a multidimensional
version of Furstenberg's correspondence principle \cite{Fu81a}, we
deduce the following consequence in combinatorics:
\begin{theorem}\label{T:MainComb}
Let  $\H$ be a Hardy field, $a_1,\ldots,a_\ell \in \G\cap \H$ be
functions with different growth rates,
 and  $\v_1,\ldots,\v_\ell\in
\Z^d$ be vectors. Suppose that the sets  $E_0, E_1, \ldots,
E_\ell\subset \Z^d$ satisfy
$$ \bar{d}(E_0\cap (E_1+k_1)\cap \cdots
\cap (E_\ell +k_\ell))=\alpha>0 $$
 for some $k_1,\ldots, k_\ell\in
\Z$. Then
$$
\liminf_{N\to\infty} \frac{1}{N}\sum_{n=1}^N \overline{d}(E_0\cap
(E_1-[a_1(n)]\v_1)\cap\cdots\cap (E_\ell-[a_\ell(n)]\v_\ell)) \geq
\alpha^{\ell+1}.
$$
\end{theorem}
Using this for $E_0=\cdots=E_\ell=E$ and $k_1=\cdots=k_\ell=0$, we
get the following strengthening of the combinatorial result
advertised in the introduction:
\begin{corollary}\label{C:MainComb1}
Let $\H$ be a Hardy field, $a_1,\ldots,a_\ell \in \G\cap \H$ be
functions with different growth rates,
 and $\v_1,\ldots,\v_\ell\in
\Z^d$ be vectors.
 Then for every  set  $E\subset \Z^d$ we have
%%Let $a_1,\ldots,a_\ell \in\mathcal{LE}$,  and suppose that every non-trivial linear combination
%%of these functions belongs in $\G$ (see \eqref{E:G}) and has super-polynomial growth.
$$
\liminf_{N\to\infty} \frac{1}{N}\sum_{n=1}^N \overline{d}(E\cap
(E-[a_1(n)]\v_1)\cap\cdots\cap (E-[a_\ell(n)]\v_\ell)) \geq
\big(\overline{d}(E)\big)^{\ell+1}.
$$
\end{corollary}
%%As an immediate  consequence we get  the result stated in the
%%introduction:
%%\begin{corollary}\label{C:MainComb2}
%%Under the assumptions of Theorem~\ref{T:MainComb},  for every  set
%%$E\subset \Z^d$ with $\bar{d}(E)>0$,
%%there exist
%% $\v \in \Z^k$  and $n\in \N$ such that
%%\begin{equation}\label{E:Conf12}
%% \v,\  \v+[a_1(n)]\v_1, \ \ldots \ ,\v+ [a_\ell(n)]\v_\ell \in E.
%%\end{equation}
%%\end{corollary}
%%In fact Theorem~\ref{T:MainComb} shows  that for every
%%$\varepsilon>0$,   for a set of $n\in\Z^d$ with  lower density at
%%least $(\bar{d}(E))^{\ell+1}-\varepsilon$, there exists $\v \in
%%\Z^k$ (depending on $n$) such that \eqref{E:Conf12} holds.

Theorem~\ref{T:MainComb} is also non-vacuous for syndetic sets
$E_0,\ldots, E_\ell\subset \N$ (in this case $\alpha$ can be as
$(\prod_{i=0}^\ell s_i)^{-1}$ where $s_i$ is the syndeticity
constant of the set $E_i$) and gives the following:
%%(one can reinterpret it as a result guarantying the
%%existence of certain polychromatic configurations )
\begin{corollary}\label{C:MainComb2}
Let $\H$ be a Hardy field and  $a_1,\ldots,a_\ell \in \G\cap \H$ be
functions with different growth rates. Let $E_0, E_1, \ldots,
E_\ell\subset \N$ be syndetic sets. Then there exist $m,n\in \N$
such that
%%Let $a_1,\ldots,a_\ell \in\mathcal{LE}$,  and suppose that every non-trivial linear combination
%%of these functions belongs in $\G$ (see \eqref{E:G}) and has super-polynomial growth.
$$
 m\in E_0,\  m+[a_1(n)]\in E_1, \ \ldots \ , m+ [a_\ell(n)]\in
 E_\ell.
$$
\end{corollary}
Corollary~\ref{C:MainComb2}   enables us to solve some non-shift
invariant systems of equations within every syndetic set. For
instance, for a syndetic set $E\subset \N$,  we can take $E_0\=cE$,
$E_i\= c_i E$, $i=1,\ldots,\ell$,  where $c,c_i$ are arbitrary
positive integers and $cE\=\{ck,k \in E\}$, and deduce that the
system of equations
\begin{align*}
c_1x_1- cx_0=&\ [a_1(n)]\\
 c_2x_2- cx_0=&\ [a_2(n)]\\
\vdots & \\
 c_\ell x_\ell- cx_0=&\ [a_\ell(n)]
\end{align*}
has a solution with $x_0, x_1, \ldots, x_\ell \in E$ and
$n\in\N$.\footnote{Similar results fail for polynomial sequences and
also fail when the set $E$ is only assumed to be piecewise syndetic.
Easy examples show that: $(i)$ If $p\in \Z[t]$ is any polynomial
different than $\pm t+c$ and $k\in\N$ is different than $1$, then
the equation $kx-y=p(n)$ has no solution with $x,y$ belonging in
some set $E$ that is an arithmetic progression. $(ii)$ If $(a(n))$
is a sequence of integers with $a(n+1)-a(n)\to\infty$ and $k\neq 1$,
then there exists a thick set $E$ such that the equation $x-ky=a(n)$
has no solution with $x,y\in E$.} Another consequence is that for
any finite coloring of $\N$, where each color class is a syndetic
set, the previous system has a  solution with   $x_0,\ldots, x_\ell$
having   arbitrary colors. In other words, if the colors classes are
denoted by $C_0,\ldots,C_k$, we can have $x_0\in C_{i_0},\ldots,
x_\ell\in C_{i_\ell}$, where $i_0,\ldots, i_\ell\in \{1,\ldots, k\}$
are arbitrary.

%%Lastly, we would like to remqark
%%\footnote{ Even for $\ell=1$ such a result fails if $c_1 \neq 1$ and
%%$a_1(t)$ is any polynomial different than $\pm t+c$. Easy examples
%%also show that for every sequence of integers $(a(n))$ and positive
%%integer $c\neq 1$, there exist a piece-wise syndetic set $E$ (i.e.
%%intersection of a syndetic set with a union of arbitrarily long
%%intervals) such that the equation $x-cy=a(n)$ has no solution with
%%$x,y\in E$ and $n\in\N$.}

  %%MAYBE REMOVE THIS We do not know whether a similar statement holds  for
%%every set of integers $E$ with positive upper density (they fail if
%%one only assumes positive upper Banch density).

%%NEXT SECTION IF OK For $\ell=1$ some positive results in this
%%direction are given in [Rivat Sarkozy].

\subsection{Key ingredients and proof plan}

\subsubsection{Key ingredients} The proof of Theorem~\ref{T:MainConv} uses the
following key ingredients:

\medskip \noindent \emph{Gowers-Host-Kra seminorms.} These are
non-negative numbers associated with every bounded measurable
function (see Section~\ref{SS:GHK}). They were defined in a
combinatorial setting in \cite{Gow01} and in an ergodic setting in
\cite{HK05a}.
 We seek to control                                                                                         the $L^2(\mu)$
 norm of the multiple ergodic averages in \eqref{E:ProductForm} by
 the seminorms of the individual functions involved.

\medskip \noindent \emph{Van der Corput's Lemma.} This elementary
estimate, and variations of it (see Section~\ref{SS:VDC}), is the
 key ingredient used
to get the desired seminorm estimates.

\medskip \noindent \emph{Decomposition results.}
These are used to
 replace sequences of the form $(f(T^nx))$
 with sequences that have more desirable properties.
We use two   decompositions, one involving dual sequences
(Proposition~\ref{P:ApprDual}), and another, much deeper one,
involving nilsequences (Theorem~\ref{P:ApprNil}). Both
decompositions originate from \cite{HK05a}.

\medskip \noindent \emph{Equidistribution results on nilmanifolds.}
These are used towards the  end of our argument when one replaces
sequences of the form $(f(T^nx))$ with nilsequences. They enable us
to carry out the finer analysis needed to prove identity
\eqref{E:ProductForm}. The equidistribution results were proved in
\cite{Fr09} using results from \cite{GT09c}  on quantitative
equidistribution of polynomial sequences on nilmanifolds.

\subsubsection{Combining the key ingredients}
Crucial to the proof of Theorem~\ref{T:MainConv} are some seminorm
estimates showing that the limit in \eqref{E:ProductForm} is $0$
when at least one of the functions involved is ``uniform enough''.
 We establish these estimates in two steps. First, we prove them for the function
that is associated with the fastest growing iterate
(Propositions~\ref{P:CharB2special} and \ref{P:CharBspecial}). This
part of the proof borrows ideas from \cite{CFH11} in order to devise
an appropriate inductive scheme (similar to the PET induction of
\cite{Be87a}) based on successive  uses of van der Corput's Lemma.
Next, we use this first step, and the decomposition result of
Proposition~\ref{P:ApprDual}, in order to replace one of the
functions with  a function  that (when evaluated in the orbit of the
corresponding transformation) gives rise to sequences (called dual
sequences) defined by a certain averaging operation. It is then
possible to devise another induction based again on successive uses
of  van der Corput's lemma and produce seminorm estimates for the
function associated with the second fastest growing iterate
(Proposition~\ref{prop:weighted}). Continuing like this, we get
seminorm estimates for  all the functions
(Proposition~\ref{pr:charcnilseq'}).

Using the seminorm estimates and the decomposition result of
Theorem~\ref{P:ApprNil}, we get  that the limit in
\eqref{E:ProductForm} remains unchanged when we replace each
function with a function that pointwise gives rise to nil-sequences.
At this advanced point in the proof, we are in position to apply
known
 equidistribution results  on nilmanifolds from
\cite{Fr09}   to complete the proof of  Theorem~\ref{T:MainConv}.

For technical reasons, complications arise  in implementing the
previous plan when one or more iterates have sub-linear growth.
 These complications are handled  using a variant of the aforementioned
 seminorm  estimates
(Proposition~\ref{T:Rn+r'})  and the equidistribution results on
nilmanifolds (Proposition~\ref{P:Rn+r}).

Recently, a relatively simple method for proving mean convergence of
the polynomial averages \eqref{E:polynomial2} was developed in
\cite{W11} (based on ideas from \cite{Ta08}),
%%In principle, this
%%method could be adapted to prove that  the limit in
%%\eqref{E:ProductForm} exists,
but up to this point it has not been successful in giving detailed
information for the limiting function. Since the precise form of the
limit  is the most crucial part of our main result, and is needed
for applications, it seems that we are forced to carry out the more
refined  analysis  summarized above.

\subsection{Further directions}
We believe that in Theorem~\ref{T:MainConv} (and its various
consequences) the restrictions we impose on the functions
$a_1,\ldots, a_\ell$ can be weakened considerably. We record here a
related problem (a special case of this already appears in
\cite{Fr10,Fr11}):
\begin{problem}\label{C:ConjProduct}
%%Let $a_1(t),\ldots,a_\ell(t)$ be functions of polynomial growth that belong to the same Hardy field.
Given a Hardy field $\H$, show that the conclusion of
Theorem~\ref{T:MainConv} holds if the functions $a_1,\ldots,
a_\ell\in \H$ have polynomial growth rate and every  non-trivial
linear combination $a(t)$ of these functions
 satisfies $|a(t)-cp(t)|/\log{t}\to\infty$ for every $c\in \R$ and
every $p\in \Z[t]$.
\end{problem}
When $\ell=1$ the result follows from the equidistribution results
in \cite{Bos94}. The problem is open even when $\ell=2$ and
$T_1=T_2$.
%%$a(t)$ of the functions $a_1(t),\ldots, a_\ell(t)$ satisfies
%%$|a(t)-cp(t)|\leq C \log{t}$ for some $c,C\in \R$ and $p\in \Z[t]$
%%with $\deg(p)\geq 2$, then \eqref{E:ProductForm} fails for some
%%system.

 When the sequences $a_1,\ldots, a_\ell$ are equal, the methods
 used in this article do not seem particularly helpful in  studying the limiting behavior of
 the averages in \eqref{E:ProductForm} (mainly because
 the seminorm estimates we use here fail in this case). We record a related problem (a special case of this already
appears in \cite{Fr10, Fr11}):
\begin{problem}\label{C:ConjCommuting1}
Let $a\colon [c,\infty)\to\R $ be a Hardy field function with
polynomial growth rate that  satisfies
$|a(t)-cp(t)|/\log{t}\to\infty$ for every $c\in \R$ and every $p\in
\Z[t]$.
%%and satisfy $t^{k+\varepsilon} \prec a_i(t)\prec t^{k+1}$ for some $k=k_i\in \N$ and $\varepsilon>0$ .
Show that   for every   system $(X,\mathcal{X},\mu,T_1,\ldots,
T_\ell)$
 and functions
   $f_1,\dots,f_\ell\in L^\infty(\mu)$ the averages
 \begin{equation}\label{E:abra}
  \frac1N \sum_{n=1}^N T^{[a(n)]}_1f_1 \cdots
T_\ell^{[a(n)]}f_\ell
 \end{equation}
converge in $L^2(\mu)$ and their limit is
 $\lim_{N\to\infty}  \frac1N \sum_{n=1}^N T^{n}_1 f_1\cdots  T^{n}_\ell f_\ell$ (this limit exists~\cite{Ta08}).
\end{problem}
The case where $T_1,\ldots, T_\ell$ are powers of a single
transformation was treated in \cite{Fr10}. In the generality stated,
the problem is open even when $\ell=2$ and $a(t)=t^{3/2}$.

Regarding pointwise convergence of the averages in
\eqref{E:ProductForm}, progress has been very scarce. The case
$\ell=1$ was treated in  \cite{BKQW05}, but other than this, even
the simplest cases remain open.
\begin{problem}\label{C:Conjpointwise}
Let $a,b$ be distinct positive  non-integers.  Show that for every
ergodic system $(X,\mathcal{X},\mu,T)$
 and functions
   $f, g \in L^\infty(\mu)$, we have
 $$
 \lim_{N\to\infty} \frac1N \sum_{n=1}^N f(T^{[n^a]}x) \cdot
g(T^{[n^b]}x)=\int f \ d\mu \cdot \int g\ d\mu
$$
for almost every $x\in X$.
\end{problem}
All  cases where both  $a$ and $b$ are greater than $1$ are open.

\subsection{Notational conventions}
The following notation will be used throughout the article:
$\N=\{1,2,\ldots\}$, $Tf=f\circ T$, $T^k=T\circ\cdots\circ T$, ${\bf
1}_E$ is the indicator function of a set $E$,  $\cC^kz$ is $z$ if
$k$ is even and $\bar{z}$ if $z$ is odd. We often write $\infty$
instead of $+\infty$.
If $a(t),b(t)$ are real valued functions
defined on some half-line $[c,\infty)$ we write $a(t)\prec b(t)$  if
$a(t)/b(t) \to 0$ as $t\to \infty$.
%%(For example,  $1\prec \log{t}\prec t^\varepsilon$ for every $\varepsilon>0$.)
We write $a(t)\ll b(t)$ if there exists $C\in \R$ such that
 $|a(t)|\leq C|b(t)|$ for all large enough $t\in \R$, and
  $a\sim b$, if $a(t)/b(t)$ converges to a nonzero constant as
$t\to\infty$.
 We denote by
 $S_ha$ the function defined by $(S_ha)(t)=a(t+h)$.
A function $a\colon
 [c,\infty)\to \R$ has degree $d$ if $t^d\ll a(t)\prec t^{d+1}$.
 By $\H$ we
 denote a translation invariant Hardy field and by $\cG$ the set of functions $a\colon
 [c,\infty)\to \R$ that satisfy $t^k\log{t}\prec a(t)\prec t^{k+1}$
 for some integer $k\geq 0$.
  If
$(X,\X,\mu,T)$ is a system, $\cI_T$ denotes the $\sigma$-algebra of
$T$-invariant sets and $\E(f|\cI_T)$ the conditional expectation on
$\cI_T$.

\section{Background Material}\label{S:Background}
In this section we put together  some background material that we
use throughout this article.
\subsection{Basic facts about Hardy fields}\label{SS:Hardy}
%%The term  Hardy field was first used by
%%Bourbaki (\cite{Bou}), they provide a convenient workisetup .
%%\begin{definition}
  Let $B$ be the collection of equivalence classes of real valued
  functions  defined on some half line $[c,\infty)$, where we
  identify two functions if they agree eventually.\footnote{The
    equivalence classes just defined are often called \emph{germs of
      functions}. We choose to use the word function when we refer to
    elements of $B$ instead, with the understanding that all the
    operations defined and statements made for elements of $B$ are
    considered only for sufficiently large values of $t\in \R$.}  A
  \emph{Hardy field} $\mathcal{H}$ is a subfield of the ring $(B,+,\cdot)$ that is
  closed under differentiation (a term first used by the
Bourbaki group in \cite{Bour61}). A \emph{Hardy field function} is a
function that belongs to some Hardy field.
 \emph{We are going to assume
throughout that all Hardy
  fields   mentioned are translation invariant},   meaning that
   if $a(t)\in \H$,
  then  $a(t+h)\in \H$ for every $h\in \R$).

%%\end{definition}

A particular example of such a Hardy field is the set $\LE$ that was
introduced by Hardy in \cite{Ha10} and consists of all
\emph{logarithmic-exponential functions}, meaning all functions
defined on some half line $(c,\infty)$ by a finite combination of
the symbols $+,-,\times, :, \log, \exp$, operating on the real
variable $t$ and on real constants.
%%that can be
%%constructed using the real constants, the functions $e^t$ and
%%$\ln{t}$, and the operations of addition, multiplication, division,
%%and composition of functions, as long as the functions constructed are
%%well defined for large $t$.
 For example   functions such as  $t^{\sqrt{2}}$,    $t(\log{t})^2$, $e^{t^2}$,
$e^{\sqrt{\log \log t}}/\log(t^2+1)$,  are all elements of $\LE$.
 Another, even more extensive example was constructed by Boshernitzan
 in \cite{Bo81}. It  satisfies the following properties:
\begin{itemize}
\item it contains the set $\LE$;

\item  it is closed under  integration; and

\item  it is closed under composition of functions that increase to infinity.
\end{itemize}

Every Hardy field function is eventually monotonic. If one of the
functions $a,b\colon [c,\infty)\to \R$ belongs to a Hardy field, and
the other function  belongs to the same Hardy field or to $\LE$,
then the  limit $\lim_{t\to\infty}a(t)/b(t)$ exists (possibly
infinite). This property is key and  will often justify our use of
l'Hopital's rule. \emph{We are going to freely use all these
properties without any further explanation in the sequel.} The
reader can find further discussion about Hardy fields  in
\cite{Bo81, Bos94}  and the references therein.

%%We caution the
%% reader  that although every function in Hardy field is
%%asymptotically comparable with every function in $\LE$,
%%  some functions in $\H$ are not comparable. This defect  can be
%%sidestepped  by  restricting
%% our attention to functions that belong to the same Hardy field.

%% the Hardy field $\LE$. Alternatively, one can work with the larger Hardy field $\mathcal{MH}$ which is defined to be %%the
%% intersection of all maximal Hardy fields (a Hardy field is maximal if there is no strictly larger Hardy field %%containing it).
%%  It can be shown  (\cite{Bo-1}, \cite{Bo0}) that the field $\mathcal{MH}$ is quite extensive: it contains $\LE$ and  it i%%s closed under
 %% integration and composition (whenever defined).
%%An example of an element in $\H$ but not in $\mathcal{MH}$ is the Gamma function  (\cite{Bo1b}).
%%{\bf Remove reference if we remove this}

\begin{definition}
We say that two functions   $a,b\colon [c,\infty)\to \R$
  have \emph{the same growth rate}, and write $a\sim b$,
if $a(t)/b(t)$ converges to a nonzero constant as $t\to\infty$. We
say that the function $a\colon [c,\infty)\to\R$ has \emph{polynomial
growth rate} if $a(t)\prec t^k$ for some $k\in \N$.
\end{definition}
Notice that if the functions $a,b$ belong to the same Hardy field,
then one of the following three alternatives holds $a\prec b$,
$b\prec a$, $a\sim b$. A key property of Hardy field functions  with
polynomial growth is that we can relate their growth rates with the
growth rates of their derivatives:
\begin{lemma}\label{L:properties}
 Let $a\colon [c,\infty)\to \R$  be a Hardy field function with  polynomial growth.
\begin{enumerate}[label=(\roman{*})]
\item
 If $a\succ 1$, then $a'\ll a/t$.

\item If  $a\succ t^\varepsilon$ for some $\varepsilon>0$, then
$a'\sim a /t$ and for every non-zero $h\in \R$ we have $S_ha-a\sim
a/t$.
\end{enumerate}
\end{lemma}
\begin{proof}
Applying l'Hopital's rule  we get
\begin{equation}\label{E:log}
\lim_{t\to\infty}\frac{ta'(t)}{a(t)}=\lim_{t\to\infty}\frac{(\log{|a(t)|})'}{(\log{t})'}=
\lim_{t\to\infty}\frac{\log{|a(t)|}}{\log{t}}.
\end{equation}
Since $a(t)$ has polynomial growth,   the last  limit is a
non-negative real number. Hence, $a'\ll a/t$.

If furthermore one has $t^\varepsilon \prec a(t)$ for some
$\varepsilon>0$ and $a(t)$ has polynomial growth,   then the
previous limit is a positive real number. This implies that $a'\sim
a/t$. Lastly, suppose that $h>0$ (a similar argument applies if
$h<0$). The mean value theorem gives that
$$
a(t+h) -a(t)=h a'(\xi_t)
$$
for some $\xi_t\in [t, t+h]$. Applying l'Hopital's rule  we get
$a'(\xi_t)/a'(t)\sim a(\xi_t)/a(t)$ and one easily sees  that
$a(\xi_t)/a(t)\to 1$. Combining the above we get $S_ha-a\sim a'$.
The proof is complete since by the first claim $a' \sim a /t$.
\end{proof}

\subsection{Basic facts from ergodic theory}
A system $(X,\X,\mu, T)$ is a Lebesgue probability space
$(X,\X,\mu)$ together with  an invertible measure preserving
transformations $T\colon X\to X$.
\subsubsection*{The ergodic theorem.} The
\emph{ergodic theorem} states that for  every system $(X,\X,\mu,T)$
 and   function $f\in L^1(\mu)$ we have  for  almost every $x\in X$
 that
$$
\lim_{N\to\infty}\frac{1}{N}\sum_{n=1}^N f(T^nx)=\tilde{f}(x)
$$
where $\tilde{f}=\E(f|\mathcal{I}_T)$ and
$$
\mathcal{I}_T\=\{A\in \X\colon \mu(T^{-1}A\triangle A)=0\}.
$$

%%\subsection{Ergodic theory}
\subsubsection*{Gowers-Host-Kra uniformity seminorms.}\label{SS:GHK}
%%We introduce a collection of seminorms that was introduced in
%%ergodic theory  by B.~Host and B.~Kra \cite{HK05a} (a combinatorial
%%variant was defined by T.~Gowers \cite{Gow01}).
Following \cite{HK05a}, where a similar definition  was given for
ergodic systems, given a    system $(X,\X,\mu,T)$ and a function
$f\in L^\infty(\mu)$, we define inductively
\begin{gather}
%%\label{E:first}
 \notag \nnorm{f}_{1,T}\=\norm{\E(f|\mathcal{I}_T)}_{L^2(\mu)} ;  \\
\label{E:recur} \nnorm f_{k+1,T}^{2^{k+1}}
\=\lim_{N\to\infty}\frac{1}{N}\sum_{n=1}^{N} \nnorm{\bar{f}\cdot
T^nf}_{k,T}^{2^{k}}.
\end{gather}
 That all limits exist and
$\nnorm{\cdot}_{k,T}$ is a seminorm can be proved as in
\cite{HK05a}. Furthermore, the limit in \eqref{E:recur} remains
unchanged if replaced
%%$\lim_{N\to\infty}\frac{1}{N}\sum_{n=1}^{N}$ c
 with the
uniform limit $\lim_{N-M\to\infty}\frac{1}{N-M}\sum_{n=M}^{N-1}$.
Using the ergodic theorem one gets
$\nnorm{f}_{1,T}^2=\lim_{N\to\infty} \frac{1}{N}\sum_{n=1}^{N} \int
\bar{f}\cdot T^nf \ d\mu $, and more generally, that
\begin{equation}\label{E:semi1}
\nnorm{f}_{k}^{2^k}= \lim_{N\to\infty} \frac{1}{N}\sum_{n_{k}=1}^N
\cdots \lim_{N\to\infty} \frac{1}{N}\sum_{n_1=1}^N \int \prod_{\e\in
\{0,1\}^{k}} \mathcal{C}^{|{\e}|}T^{{\e}\cdot \n}f \ d\mu,
\end{equation}
where  $\n=(n_1,\ldots,n_{k})$ and  for $\e\in \{0,1\}^{k}$ we let
$$
\n\cdot \e\=n_1\epsilon_1+\cdots +n_{k}\epsilon_{k}, \ \
|\e|=\epsilon_1+\cdots + \epsilon_{k},
$$
and for $z\in \C$ and $k$ nonnegative integer we let
$$
\cC^{k}z\= \begin{cases} z  &\text{ if } k \text{ is even} \\
 \bar{z} &\text{ if } k \text{ is odd.} \end{cases}
$$
It follows from Theorem 13.1 in \cite{HK05a} that in \eqref{E:semi1}
the  iterative limit
%%$\lim_{N\to\infty} \frac{1}{N}\sum_{n_{k}=1}^N \cdots \lim_{N\to\infty} \frac{1}{N}\sum_{n_1=1}^N $
can be replaced with the limit $\lim_{N\to\infty}
\frac{1}{N^k}\sum_{1\leq n_1,\ldots, n_k\leq N}$.
 Using \eqref{E:semi1} and the ergodic theorem
 one can check  that
\begin{equation}\label{E:seminonergodic}
 \nnorm{f\otimes\overline{f}}_{k,T\times T}\leq  \nnorm{f}_{k+1,T}^2
\end{equation}
holds for every $k\in \N$. We also remark that $\nnorm{f}_{k,T}\leq
\nnorm{f}_{k+1,T}$ holds for every $k\in \N$.

\subsection{Dual functions, dual sequences, and  weak decomposition}
\subsubsection{Dual functions}
 Let $(X,\X,\mu,T)$ be a system,  $f\in L^\infty(\mu)$, and
 $M\in\N$. We define
$$
A_M(f)\=\frac{1}{M^k}
\sum_{\m\in [1,M]^k} \prod_{\substack{\e\in \{0,1\}^k,\\
\epsilon \neq 00\cdots 0}}
 \cC^{|\e|}T^{\m\cdot \e}f.
$$
 It is shown in \cite{HK05a} that the averages $A_M(f)$ converge
in $L^2(\mu)$  and in \cite{As10} that they converge pointwise. We
define
$$
  \cD_{k,T} f\=\lim_{M\to\infty}A_M(f)
$$
and call any such function a \emph{level $k$ dual function}.
 For instance, we have
$$
(\cD_{2,T}f)(x)=\lim_{M\to\infty}\frac{1}{M^2} \sum_{1\leq
m_1,m_2\leq N} T^{m_1}\bar{f}\cdot T^{m_2}\bar{f}\cdot T^{m_1+m_2}f.
$$
%%$$
%%(\cD_2f)(x)=\lim_{N\to\infty}\frac{1}{N} \sum_{1\leq n_2\leq N}
%%\lim_{N\to\infty}\frac{1}{N} \sum_{1\leq n_1\leq N} f(T^{n_1}x)\cdot
%%f(T^{n_2}x)\cdot f(T^{n_1+n_2}x).
%%$$

The importance of dual functions in the current article stems from
the following result (it follows from \eqref{E:semi1} and the fact
that the iterative limit can be replaced with a limit over cubes):
\begin{proposition}\label{P:dual}
Let $(X,\X,\mu,T)$ be a system.  Then for every $f\in L^\infty(\mu)$
and  $k\in\N$ we have
$$
\int f \cdot  \cD_{k,T} f \ d\mu =\nnorm{f}_{k,T}^{2^k}.
$$
\end{proposition}
As a consequence,   $\nnorm{f}_{k,T}\neq 0$ if and only if  $f$
positively  correlates with  some dual function of level $k$.
\subsubsection{Dual sequences}
 A\emph{dual sequence of level $k$} is a sequence
$(\cD(n))$ of the form
$$
\cD(n)\=\lim_{M\to\infty} \frac{1}{M^k}
\sum_{\m\in [1,M]^k } \prod_{\substack{\e \in \{0,1\}^k,\\
\epsilon \neq 00\cdots 0}}
 \cC^{|\e|} d(n+\m\cdot \e),
$$
where $(d(n))$ is a bounded sequence such that  the above limit
exists for every $n\in \N$.

For future use, we record  the identity
\begin{equation}\label{E:alternate}
\cD(n)=\lim_{M\to\infty} \frac{1}{M^k}
\sum_{\m\in [1,M]^k} \prod_{\substack{\e \in \{0,1\}^k,\\
\epsilon \neq 00\cdots 0}} \cC^{|\e|} \d_{\e}(\m +n\ \!\tilde{\e})
\end{equation}
where $\tilde{\e}$ is any vector in $\{0,1\}^k$ such that $\e\cdot
\tilde{\e}=1$ and
$$
\d_{\e}(\m)=d(\e\cdot \m).
$$
For instance, if $(\cD(n))$ is a dual sequence of level $2$, then
\begin{align*}
\cD(n)=&\lim_{M\to\infty}\frac{1}{M^2} \sum_{1\leq m_1,m_2\leq N}
\bar{d}(n+m_1)\cdot \bar{d}(n+m_2)\cdot d(n+m_1+m_2)\\
=&\lim_{M\to\infty}\frac{1}{M^2} \sum_{1\leq m_1,m_2\leq N}
\bar{\d}_1(m_1+n,m_2) \cdot \bar{\d}_2(m_1,m_2+n) \cdot
\d_3(m_1+n,m_2),
\end{align*}
where
$$
\d_1(m_1,m_2)\=d(m_1), \quad  \d_1(m_1,m_2)\=d(m_2), \quad
\d_3(m_1,m_2)\=d(m_1+m_2).
$$
\subsubsection{Weak decomposition}
For the purpose of this article the significance of the collection
of dual sequences stems from the following decomposition result:
\begin{proposition} [Weak decomposition]\label{P:ApprDual}
Let $(X,\X,\mu,T)$ be a system, $f\in L^\infty(\mu)$,  and $k\in
\N$. Then for every $\varepsilon>0$, there exist functions $f_s,
f_u, f_e\in
L^\infty(\mu)$, %%with $L^\infty(\mu)$ norm at most
%%$2\norm{f}_{L^\infty(\mu)}$,
 such that
\begin{enumerate}
\item $f=f_s+f_u+f_e$;
\item $\nnorm{f_u}_{k}=0$; \  $\norm{f_e}_{L^1(\mu)}\leq \varepsilon$; \ and
\item $f_s=\sum_{i=1}^m c_i\ \! f_{s,i}$, where  $c_i\in \R$,   $f_{s,i}\in L^\infty(\mu)$,
 and for  almost
every $x\in X$ the sequence $(f_{s,i} (T^nx))_{n\in\N}$ is a  dual
sequence of level   $k$.
\end{enumerate}
\end{proposition}
\begin{proof}
Let $\varepsilon>0$, $k\in \N$,  and $f\in L^\infty(\mu)$.  We
construct an  invariant sub-$\sigma$-algebra $\CZ_{k-1}$ of $\X$
exactly as in Section~4 of \cite{HK05a}. It satisfies the same
property as in Lemma~4.3 of \cite{HK05a}, namely,
\begin{equation}\label{E:DefZ_l}
 \text{\em for } f\in L^\infty(\mu),\
\E(f|\cZ_{k-1})=0\text{ \em if and
  only if }\  \nnorm f_{k} = 0.
\end{equation}
%%Equivalently, one has
%%\begin{equation}\label{E:DefZ_l'}
%%L^\infty(\cZ_{k-1},\mu)=\Big\{f\in L^\infty(\mu)\colon \int f\cdot g
%%\ d\mu=0
%% \text{ for every } g\in L^\infty(\mu) \text{ with }
%%\nnorm g_{k} = 0\Big\}.
%%\end{equation}

We can decompose $f$ as  $f=f_u+g$ where $g=\E(f|\mathcal{Z}_{k-1})$
and $f_u \bot\ \! L^\infty(\mathcal{Z}_{k-1}, \mu)$.
%%with $\norm{g}_{L^\infty(\mu)}\leq norm{f}_{L^\infty(\mu)}$.
It follows from \eqref{E:DefZ_l} that  $\nnorm{f_u}_{k}=0$. It is
clear that $f_u, g, \in L^\infty(\mu)$.
%% $\norm{g}_{L^\infty(\mu)}\leq \norm{f}_{L^\infty(\mu)}$
%%and $\norm{f_u}_{L^\infty(\mu)}\leq 2\norm{f}_{L^\infty(\mu)}$.

We  claim that linear combinations of dual functions of level $k$
are dense in $L^1(\cZ_{k-1}, \mu)$.  Indeed,  by duality, it
suffices to show that if $\tilde{f}\in L^\infty(\cZ_{k-1},\mu)$
satisfies $\int \tilde{f} \cdot \cD_{k,T}f\ d\mu=0$ for every $f\in
L^\infty(\mu)$, then $\tilde{f}=0$.
 Taking $f=\tilde{f}$ gives  $\int \tilde{f}\cdot \cD_{k,T}\tilde{f}\ d\mu=0$,  and
by Proposition~\ref{P:dual} we get $\nnorm{\tilde{f}}_k=0$. Since
$\tilde{f}\in L^\infty(\mathcal{Z}_{k-1},\mu)$, we deduce from
\eqref{E:DefZ_l}
 that $\tilde{f}=0$. This completes the proof of the claim.

Keeping in mind that  $g\in L^\infty(\cZ_{k-1},\mu)$, the claim
enables us to decompose $g$ as $g=f_s +f_e$, where $f_s$ is a finite
linear combination of dual functions of level $k$ and
$\norm{f_e}_{L^1(\mu)}\leq \varepsilon$. Since  the function $g$ and
all dual functions are bounded,   the function $f_e$ is
 bounded. The proof ends upon noticing that if $ h$ is a dual
function of level $k$, then for   almost every $x\in X$ the sequence
$( h(T^nx))_{n\in\N}$ is a dual sequence of level $k$.
\end{proof}

\subsection{Nilsystems, nilsequences, and strong decomposition }
A \emph{nilmanifold}  is a homogeneous space $X=G/\Gamma$ where  $G$
is a nilpotent Lie group, and $\Gamma$ is a discrete cocompact
subgroup of $G$. If $G_{k+1}=\{e\}$ , where $G_k$ denotes the $k$-th
commutator subgroup of $G$, we say that $X$ is a
 $k$-\emph{step nilmanifold}.

 A $k$-step nilpotent Lie group $G$ acts on $G/\gG$ by left
translation where the translation by a fixed element $a\in G$ is
given by $T_{a}(g\gG) = (ag) \gG$.  By $m_X$ we denote the unique
probability measure on $X$ that is invariant under the action of $G$
by left translations (called the {\it normalized Haar measure}), and
by $\G/\gG$ we denote the Borel $\sigma$-algebra of $G/\gG$. Fixing
an element $a\in G$, we call the system $(G/\gG, \G/\gG, m_X,
T_{a})$ a {\it $k$-step  nilsystem}. The reader can find more
material about nilmanifolds in \cite{Lei05a} and the references
therein.

If $X=G/\Gamma$ is a $k$-step nilmanifold,  $a\in G$,  $x\in X$, and
$f\in C(X)$,
 we call the sequence $(f(a^nx))_{n\in\N}$ a \emph{basic $k$-step nilsequence}.
A \emph{$k$-step nilsequence}, is a uniform limit of \emph{basic
$k$-step nilsequences}.
%%As is easily verified, the collection of
%%$k$-step nilsequences, with the topology of uniform convergence,
%%forms a closed algebra.
%% We caution the reader that in other articles the term $k$-step nilsequence is used for what we call here
%% basic $k$-step nilsequence, and in some instances the function $f$ is assumed to satisfy weaker or stronger
%%conditions than continuity.

\subsubsection{Strong decomposition}
  The next decomposition
result will be crucial for our study. For ergodic systems it is a
direct consequence of  a deep  structure theorem in \cite{HK05a};
the extension to the  non-ergodic case was treated in \cite{CFH11}
(see Proposition~3.1).
\begin{theorem}[Strong decomposition] \label{P:ApprNil}
Let $(X,\X,\mu,T)$ be a system, $f\in L^\infty(\mu)$,  and $k\in
\N$. Then for every $\varepsilon>0$, there exist  functions $f_s,
f_u, f_e\in L^\infty(\mu)$, with $L^\infty(\mu)$ norm at most
$2\norm{f}_{L^\infty(\mu)}$, such that
\begin{enumerate}
\item $f=f_s+f_u+f_e$;
\item $\nnorm{f_u}_{k+1}=0$; \  $\norm{f_e}_{L^2(\mu)}\leq \varepsilon$; \ and
\item for  almost every $x\in X$  the sequence $(f_s (T^nx))_{n\in\N}$
is a $k$-step nilsequence.
\end{enumerate}
\end{theorem}

\subsection{The van der Corput Lemma}\label{SS:VDC}
 A key tool in proving uniformity estimates  is  the following variant of van
der Corput's
 fundamental estimate (proved as in  Lemma~3.1 in  \cite{KN74}):
\begin{lemma} \label{L:VDC2} Let $N\in \N$ and  $v_1,\ldots, v_N$ be vectors in an inner product space.
Then for every integer  $H$ between $1$ and $N$ we have ($\Re(z)$
denotes the real part of a complex number $z$)
$$
\norm{\frac{1}{N}\sum_{n=1}^{N} v_n}^2\leq \frac{2}{H}\sum_{h=1}^{H}
\big(1-\frac{h}{H}\big) \Re\Big(\frac{1}{N}\sum_{n=1}^{N} < v_{n+h},
v_n>\Big) + \frac{2}{H} + \frac{4H}{N}  .
$$
%%where $|e_{R,N}|\leq 4\cdot (R^{-1}+RN^{-1})$.
\end{lemma}
We also use the following   qualitative variant:
\begin{lemma}\label{L:N-VDC}
Let  $(v_n)$ be a bounded  sequence of vectors in an inner product
space, and $(\Phi_N)$ be a F{\o}lner sequence of subsets of $\N$.
Then
%%$$
%%\overline{\lim}_{H\to\infty}\E_{1\leq h\leq H} b_h=0.
%%$$
$$
\limsup_{N\to\infty}
\norm{\frac{1}{|\Phi_N|}\sum_{n\in\Phi_N}v_n}^2\leq 4 \ \!
\limsup_{H\to\infty}\frac{1}{H}\sum_{ h=1}^H
\limsup_{N\to\infty}\Big|
\frac{1}{|\Phi_N|}\sum_{n\in\Phi_N}<v_{n+h},v_{n}>\Big|.
$$
\end{lemma}
In most cases we apply this lemma for $\Phi_N=[1,N]$, $N\in \N$.
\section{Seminorm estimates for the highest degree iterate: Two transformations}\label{S:Char2}
 An important step towards establishing Theorem~\ref{T:MainConv} is  to obtain estimates that enable us to
 control the  $L^2(\mu)$ norm of
 the averages in \eqref{E:ProductForm} by the uniformity seminorms
of the individual functions.
  In this section and the next one, our goal is  to do this  for
  the function that is associated with the
 fastest growing iterate. In subsequent sections we utilize this
 information in order to get similar estimates
 for  the other functions.

Since the
 proof is notationally heavy, we choose to first present it in detail for
  the case of two commuting transformations. The argument that covers the general
 case  is very similar and we  sketch its proof in the next
 section.

The main goal in this section is to establish the following result:
\begin{proposition}\label{P:CharB2special}
Let $(X,\X,\mu,T_1 ,T_2)$ be a system and $f_1,f_2\in L^\infty(\mu)$
be functions. Let $\H$ be a Hardy field, $a_1,a_2\in \G\cap \H$ be
functions that satisfy $a_1\succ a_2$, and let $d\=\deg(a_1)$ (all
notions are defined in Section~\ref{subsec:families}). Then there
exists $k=k(d)$ such that: If $\nnorm{f_1}_{k,T_1}=0$, then the
averages
$$
\frac{1}{N}\sum_{n=1}^{N}  T_1^{[a_1(n)]}f_1\cdot T_2^{[a_2(n)]}f_2
$$
converge to $0$ in $L^2(\mu)$.
\end{proposition}
Our method necessitates that we prove a more general result that we
 present next.
 %% This  more general result is also needed
%%when one tries to get seminorm estimates for the $L^2(\mu)$ norm of
%%the averages ??? when  $\ell\geq 3$.
\begin{proposition}\label{P:CharB2}
Let $(X,\X,\mu,T_1,T_2)$ be a system and $f_1,\ldots, f_m\in
L^\infty(\mu)$ be functions.  Let   $(\A,\B)$  be a nice
 ordered  family of pairs of functions   with degree $d$
 (all notions are defined in Sections \ref{subsec:families} and \ref{subsec:nicefamilies}).
 Then there exists $k=k(d,m)\in \N$ such that: If      $\nnorm{f_1}_{k,T_1}=0$,
 then
\begin{equation}\label{E:fgh}
\lim_{N\to\infty}\sup_{E\subset \N}\norm{\frac{1}{N}\sum_{n=1}^{N}
\prod_{i=1}^m T_1^{[a_i(n)]} T_2^{[b_i(n)]}f_i \cdot {\bf
1}_{E}(n)}_{L^2(\mu)}=0.
\end{equation}
\end{proposition}
Applying this result to the nice family $(\A,\B)$  defined by
$\A\=(a_1,0)$ and $\B\=(0,a_2)$,  one sees that
Proposition~\ref{P:CharB2special} follows from
Proposition~\ref{P:CharB2}.

\subsection{Families of pairs of functions  and their type}\label{subsec:families}

\subsubsection{Degree and equivalence}

\begin{definition}
If  $a\colon [c,\infty)\to \R$ is a function with polynomial growth
rate, and $k_0$ is the smallest non-negative integer $k$ such that
$a(t)\prec t^k$, we define $d\=k_0-1$ to be the  \emph{degree} of
the function, and  write $\deg(a)=d$.
\end{definition}
As a consequence, $\deg(a)=-1$ if and only if $a(t)\to 0$, and
$\deg(a)=d\geq 0$ if and only if $t^d\ll a(t)\prec t^{d+1}$. For
example,  $\deg(1/t)=-1$, $\deg(1)=\deg(\sqrt{t})=\deg(t/\log t)=0$,
$\deg(t)=\deg(t^{1.5})=1$.

We remind the reader that two functions  $a, b\colon [c,\infty)\to
\R$ have the same growth rate, in which case   we write $a\sim b$,
if $a(t)/b(t)$ converges to a  non-zero constant as $t\to\infty$. We
will make use of the following  stronger notion of growth
equivalence:
\begin{definition}
We say that two functions   $a,b\colon [c,\infty)\to \R$
   are \emph{equivalent},   and write $a\cong b$,  if they have polynomial growth rate
   and satisfy
$\deg(a-b)< \min\{\deg(a),\deg(b)\}$.
\end{definition}
 Notice that  if $a\cong b$, then $a(t)/b(t)\to 1$,
 but the converse is not true. For example  $t^{1.5} \ncong t^{1.5}+t^{1.1}$.

\subsubsection{Families of pairs of functions}
 Let $m\in\N$. Given two ordered families of functions
 $$
 \A\=(a_1,\ldots,a_m), \quad \B\=
 (b_1,\ldots,b_m),
 $$
where $a_i,b_i\colon [c,\infty)\to \R$ have polynomial growth rate,
we define the \emph{ordered family of pairs of functions} $(\A,\B)$
as follows
$$
(\A,\B)\=\big((a_1,b_1),\ldots,(a_m,b_m)\big).
$$
The reader is advised to think of this family of pairs as an
efficient way to record  the functions that appear in the iterates
\eqref{E:fgh}.

The maximum of the degrees of the functions in the families $\A$ and
$\B$ is called the \emph{degree of the family} $(\A,\B)$.

For convenience of exposition, if pairs of bounded functions  appear
in $(\A,\B)$ we remove them, and henceforth we assume:
\begin{itemize}
\item[] \em
All families $(\A,\B)$ that we consider do not contain pairs of
bounded functions.
\end{itemize}

\subsubsection{Definition of type}\label{SS:type}
 We fix a non-negative integer $d$ and
restrict ourselves to families $(\A,\B)$ with degree between $0$ and
$d$.

 Let
\begin{equation} \label{E:A'}
\A'\=\{a\in \A\colon  a \text{ is not bounded} \}.
\end{equation}
and
\begin{equation} \label{E:B'}
\B'\=\{b_i\in \B\colon  a_i \text{ is bounded} \}.
\end{equation}

For $i=0,1,\ldots, d$, let $w_{1,i}$, $w_{2,i}$ be the number of
distinct non-equivalent classes of polynomials of degree $i$ in
$\A'$ and  $\B'$ correspondingly  (if $\B'$ is empty, then
$w_{2,i}=0$ for $i=0,1,\ldots, d$).

We define the \emph{(matrix) type} of the family $(\A,\B)$ to be the
$2\times (d=1)$  matrix
$$
\begin{pmatrix}
w_{1,d}& \ldots &  w_{1,0}\\ w_{2,d}& \ldots& w_{2,0}
\end{pmatrix}.
$$
 For example, consider the family of pairs
$$
\big((t^{2.5}, t^{3.5}),\ (t^{2.5}+t^2,t), \ (t^{2.5}+t^{1.5}, 2t),\
(t^{0.5}, t) \ ((t+1)^{0.5}-t^{0.5},t^{1.5}), \ (0,t^{0.5})\big).
$$
Then $d=3$, $\A'=\{t^{2.5},t^{2.5}+t^2,t^{2.5}+t^{1.5}, t^{0.5}\}$,
and $\B'=\{t^{1.5}, t^{0.5}\}$. As a consequence, the family of
pairs $(\A,\B)$ has type
$$
\begin{pmatrix}
0& 2& 0& 1\\ 0& 0& 1 & 1
\end{pmatrix}.
$$
We order the set of all possible types lexicographically; we  start
by comparing the first element of the first row of each matrix, and
after going through all the elements of the first row,  we compare
the  elements of the second row of each matrix,  and so on. In other
words: given two $2\times (d+1)$ matrices $W\=(w_{i,j})$ and
$W'\=(w'_{i,j})$, we say that $W \succ W'$ if:  $w_{1,d}>w'_{1,d}$,
or $w_{1,d}=w'_{1,d}$ and $w_{1,d-1}>w'_{1,d-1}$, $\ldots$, or
$w_{1,i}=w'_{1,i}$ for $i=0,\ldots,d$ and $w_{2,d}>w'_{2,d}$, and so
on.

As an example we mention
\begin{multline*}
\begin{pmatrix}
 2& 2\\ 0& 0
\end{pmatrix}
\succ
\begin{pmatrix}
 2& 1\\ \star &\star
\end{pmatrix} \succ
\begin{pmatrix}
 2& 0\\ \star & \star
\end{pmatrix} \succ
\begin{pmatrix}
 1& \star\\ \star& \star
\end{pmatrix} \succ
\begin{pmatrix}
 0& \star\\ \star& \star
\end{pmatrix}
%%\\
\succeq
\begin{pmatrix}
 0& 0\\ \star& \star
 \end{pmatrix}\succeq
 \begin{pmatrix}
 0& 0\\ 0& \star
\end{pmatrix}
\succeq
 \begin{pmatrix}
 0& 0\\ 0& 0
\end{pmatrix}
\end{multline*}
where in the place of the stars one can put any collection of
non-negative integers.

An important observation is that although for a given type $W$ there
is an infinite number of possible types $W'$ that are smaller than
$W$, we have
\begin{lemma}
\label{lem:decreasing} Every decreasing sequence of
 types of families of pairs   is eventually stationary.
\end{lemma}
Therefore, if some operation reduces the type of a certain family of
pairs of functions, then after a finite number of repetitions it
will terminate.

\subsection{Nice families and the van der Corput operation}\label{subsec:nicefamilies}
 In this subsection we define a class of ``nice''  families of pairs of
 functions that will be instrumental for our subsequent discussion.
Furthermore, we define an operation that sends nice families to nice
families and reduces their type.
\subsubsection{Nice families}
We remind the reader of our  definition of the class of good
functions
$$
\cG=\{a\colon [c,\infty)\to \R \text{ such that } t^d \log{t} \prec
a(t)\prec t^{d+1} \text{ for some  integer } d\geq 0 \}.
%%\text{ and } \varepsilon>0\}\cup \{a\inC(\R_+)\colon a(t)\to \text{const}\}.
$$
%%We also define the following class
%%$$
%%\cG_0\=\cG \cup \{ a\in C(\R_+)\colon a(t)\to 0\}.
%%$$

\begin{definition}
Given a function $a\colon [c,\infty)\to\R$, we  define $\cF(a)$ to
be the family of functions  that contains all integer combinations
of shifts of $a$, meaning,
$$
\cF(a)\=\Big\{\sum_{i=1}^l k_i \cdot S_{h_i}a, k_i\in \Z, h_i,l\in
\N\Big\}.
$$
\end{definition}
Using Lemma~\ref{L:properties}, one sees that if $a\in \G$ and $b\in
\F(a)$, then either $b(t)\to 0$ or $ b\in \G$.

 Henceforth, we are going to work with the following class  of pairs of
functions:
\begin{definition}\label{D:good} Let $\H$ be a Hardy field,
$a,b\in \cG\cap \H$,   $a_i\in \cF(a)$
  and  $b_i \in \cF(b) $  for $i=1,\ldots,m$,  and
 $\cA\=(a_1,\ldots,a_m)$, $\cB\=(b_1,\ldots,b_m)$.
We call the ordered  family  $(\cA,\cB)$ \emph{nice} if
\begin{enumerate}

\item\label{it:nice1}
  $a_1-a_i\succ 1 $ and $a_i \ll a_1 $  for $i=2,\ldots,m$;
\medskip

\item\label{it:nice2}
  $b_i\prec a_1$ for $i=1,\ldots,m$;
\medskip

\item \label{it:nice3}
   $ b_1-b_i \prec a_1-a_i $ for $i=2,\ldots,m$.
\end{enumerate}
\end{definition}
For example, if $\H$ is a Hardy field, and   $a,b\in \cG\cap \H$
satisfy $a\succ b$, then the ordered family of pairs $\big((a,0),
(0,b)\big)$ is nice. If in addition we assume that $\deg(b)\geq 1$,
then also the family $\big((a,-b), (S_ha,-b), (0,S_hb-b)\big)$ is
nice for every $h\in \N$. This is a special case of a more general
phenomenon that will be explained in Section~\ref{SS:Reducing}.

\subsubsection{The van der Corput operation}
\label{subsec:vdcoperation} Given an ordered family of   pairs of
functions $(\A,\B)$, a pair of functions $(a,b)$, and $h\in\N$,
 we define the following operation
  \begin{multline*}
(a,b,h)\vdc(\A,\B) \= \\ \big(
(S_ha_1-a,S_hb_1-b),\ldots,(S_ha_m-a,S_hb_m-b),
 %%\\
(a_1-a,b_1-b),\ldots,(a_m-a,b_m-b)\big)^\ast.
\end{multline*}
where  $^\ast$ is the operation that removes all pairs of bounded
functions.

\subsection{Strategy of proof of Proposition~\ref{P:CharB2}}\label{SS:Example1}
Our proof strategy  of Proposition~\ref{P:CharB2} is to successively
apply Lemma~\ref{L:N-VDC} in order to bound  the $L^2(\mu)$ norm of
the averages in question with the $L^2(\mu)$ norm of averages that
are simpler to deal with. In order to carry out this reduction a key
step is to show that given a nice
 family of pairs $(\A,\B)$ with $\deg (a_1)\geq 1$,  it is always  possible  to find
 $(\tilde{a},\tilde{b})\in (\A,\B)$ such that for all large
enough $h\in \N$ the
 operation $(\tilde{a},\tilde{b},h)\vdc$ leads to a  nice family of pairs
 that has smaller type.
Eventually, this procedure leads to families of pairs with
sub-linear growth (i.e. with degree $0$), in which case
Proposition~\ref{P:CharB2} can be established directly in a
relatively simple manner.

We explain how this reduction to the degree $0$  case works in the
next example:

\begin{example*}
\label{SSS:example2} Our goal is to find $k\in \N$ such that if
$\nnorm{f_1}_{k,T_1}=0$, then the averages
 \begin{equation}\label{E:1511}
\frac{1}{N}\sum_{n=1}^{N} f_1(T_1^{[n^{1.5}]}x)\cdot
f_2(T_2^{[n^{1.1}]}x)
 \end{equation}
 converge to $0$ in $L^2(\mu)$ as $N\to\infty$.

   We define $\A=(t^{1.5},0)$, $\B=(0,t^{1.1})$, and  introduce the following nice family of pairs
   of functions
 $$
 (\A,\B)=\big((t^{1.5},0),(0,t^{1.1})\big).
 $$
 This family is nice and has  type
 $\left(
\begin{smallmatrix}
1& 0\\ 1&  0
\end{smallmatrix}\right)$.
  Applying the vdC operation with    $(a,b)=(0,t^{1.1})$, we see that for
   $h\in\N$,  the  ordered family  $(a,b,h)\vdc(\A,\B)$ is equal to
 $$
 \big( \big((t+h)^{1.5},-t^{1.1}\big),\big(0,(t+h)^{1.1}-t^{1.1}\big), \big(t^{1.5},-t^{1.1}\big)\big).
 $$
The important point is that for every $h\in\N$  this new family  is
also nice and has smaller type, namely $\left(\begin{smallmatrix} 1&
0\\ 0& 1
\end{smallmatrix}\right)$.
Loosely speaking, one expects to be able to show (using
Lemma~\ref{L:N-VDC}) that the averages \eqref{E:1511} converge to
$0$ in $L^2(\mu)$ once one can show that for every $h\in \N$ the
averages
 $$
 \frac{1}{N}\sum_{n=1}^{N}
  f_1(T_1^{[(n+h)^{1.5}]}T_2^{[-n^{1.1}]}x)\cdot
\tilde{g}(T_2^{[(n+h)^{1.1}-n^{1.1}]}x)
 \cdot \tilde{h}(T_1^{[n^{1.5}]}T_2^{[-n^{1.1}]}x)
 $$
  converge to $0$
in $L^2(\mu)$
 for all $\tilde{g},\tilde{h}\in
L^\infty(\mu)$.

For $h\in \N$, applying the vdC operation one more time with
$(a,b)=(0,(t+h)^{1.1}-t^{1.1})$ leads to  a nice  ordered family
with $4$ pairs and type $\left(\begin{smallmatrix} 1& 0\\ 0& 0
\end{smallmatrix}\right)$. Lastly, for $h\in \N$, applying the vdC operation one more time with
$(a,b)=(t^{1.5},-(t+h)^{1.1})$, it is easy to see that we  get a
nice
ordered family  with $7$ pairs and  type $\left(\begin{smallmatrix} 0& 7\\
0& 0
\end{smallmatrix}\right)$. In this case all functions involved have
sub-linear growth, and the iterates of $T$ grow faster than any of
the iterates of $S$.  Taking advantage of this fact, we can show in
a relatively simple way that the corresponding multiple ergodic
averages converge to $0$ in $L^2(\mu)$ if $\nnorm{f}_{16,T_1}=0$.
\end{example*}

\subsection{Two technical lemmas} We establish two simple results
that will be used repeatedly.
\begin{lemma}\label{L:Key1}
Let   $a\colon [c,\infty)\to \R$ be a Hardy field function with
non-negative degree $d$ and let $b\in \cF(a)$. Then either $b(t)\to
0$, or there exists $k\in \{0,\ldots, d\}$ such that $b\sim a/t^k$.
\end{lemma}
\begin{proof}
Without loss of generality we can assume that $a(t)\to \infty$.
Suppose that
$$
b=\sum_{i=1}^l k_i \cdot S_{h_i}a.
$$
Since $\deg (a)=d$ we have by Lemma~\ref{L:properties} that
$a^{(d+1)}(t)\to 0$. Using this and Lagrange's remainder formula for
the Taylor series of the function $a(t)$, we see that for  $h\in \N$
we have
$$
S_ha=\sum_{i=0}^d a^{(i)}\  h^i/i!+ e_h
$$
where $e_h\colon [c,\infty)\to \R$ is a function that satisfies
$e_h(t)\to 0$. Combining the above identities we deduce that
$$
b= \sum_{i=0}^d c_i a^{(i)} +e
$$
for some constants $c_i\in \R$ and function $e\colon [c,\infty)\to
\R$ that satisfies $e(t)\to 0$. If $c_i=0$ for $i=0,\ldots, d$, then
$b(t)\to 0$. Otherwise, let $i_0$ be the smallest $i$ such that
$c_{i}\neq 0$. Then $b\sim a^{(i_0)}$, and by
Lemma~\ref{L:properties} we have $a^{(i_0)}\sim a/t^{i_0}$. Taking
$d=i_0$ completes the proof.
\end{proof}

\begin{lemma}\label{L:KeyHardy}
Let   $a\colon [c,\infty)\to \R$ be a Hardy field function with
polynomial growth rate and $a_1,a_2\in \cF(a)$ be such that
$a_1\succ t^{\varepsilon}$ for some $\varepsilon>0$ and $a_2\ll
a_1$.

\begin{enumerate}[label=(\roman{*})]
\item  If $a_1\ncong a_2$, then  $S_ha_1-a_2 \sim
 a_1$ for every non-zero $h\in \R$.

\item If  $a_1\cong a_2$, then $S_ha_1-a_2\ll a_1/t$ for every $h\in
\R$, and $S_ha_1-a_2\sim a_1/t$ for all but one $h\in \R$.
\end{enumerate}
\end{lemma}
\begin{remark}The assumption $a_1,a_2\in \cF(a)$ is necessary. For $(i)$ take
$a_1(t)=t^{1.5}+t^{1.1}, a_2(t)=t^{1.5}$, and for $(ii)$ take
$a_1(t)=t^{1.5}+t^{0.9}, a_2(t)=t^{1.5}$.
\end{remark}
\begin{proof}
We prove $(i)$. Suppose on the contrary that  $S_ha_1-a_2\nsim a_1$
for some $h\in \R$. Since $S_ha_1-a_2\ll a_1$, we deduce that
$S_ha_1-a_2\prec a_1$.

We claim that  $S_ha_1-a_2\ll a_1/t$. Indeed,  by Lemma~\ref{L:Key1}
we have $a_1\sim a/t^k$ for some non-negative integer $k$.  Since
$S_ha_1-a_2 \in \cF(a)$, Lemma~\ref{L:Key1} gives that either
$S_ha_1-a_2\prec 1$, or $S_ha_1-a_2\sim a/t^{k'}$ for some
non-negative integer $k'$. If $S_ha_1-a_2\prec 1$, then the claim is
proved because $\deg (a_1)\geq 1$. If $S_ha_1-a_2\sim a/t^{k'}$,
then since  $S_ha_1-a_2\prec a_1\sim a/t^{k}$, we deduce that
$k'>k$, proving the claim.

Using the previous claim, Lemma~\ref{L:properties}, and expressing
$a_1-a_2$ as $(a_1-S_ha_1)+(S_ha_1-a_2)$, we deduce that $a_1-a_2\ll
a_1/t$. This is a contradiction since by assumption   $a_1\ncong
a_2$.

We prove $(ii)$. Expressing  $S_ha_1-a_2$ as $(S_ha_1-a_1)+
(a_1-a_2)$ and using Lemma~\ref{L:properties} and our assumption
$a_1\cong a_2$, we see that for every $h\in \R$ we have
$S_ha_1-a_2\prec a_1$. From this we deduce  as in the proof of part
$(i)$ that $S_ha_1-a_2\ll a_1/t$ for every $h\in \R$. It remains to
show that if $S_{h_0}a_1-a_2\prec a_1/t$, then $S_{h}a_1-a_2\sim
a_1/t$ for every $h\neq h_0$. To see this, we express $S_{h}a_1-a_2$
as $(S_ha_1-S_{h_0}a_1)+(S_{h_0}a_1-a_2)$, and use that by
Lemma~\ref{L:properties} we have $S_{h}a_1-a_1\sim a_1/t$ for every
non-zero $h\in \R$. This completes the proof.
\end{proof}

\subsection{Reducing the type}\label{SS:Reducing} The next lemma is a
key ingredient of  the proof of Proposition~\ref{P:CharB2}.

\begin{lemma}\label{L:reduceA}
Let $(\A,\B)$ be  a nice family of pairs of functions,  and suppose
that $\deg(a_1)\geq 1$.
 Then there exist
 %%for some appropriate choice of
  $\tilde{a}\in \cA\cup\{0\}$ and $\tilde{b}\in \cB$, such that for every large enough $h\in\N$,  the
 family $(\tilde{a},\tilde{b},h)\vdc(\A,\B)$   is nice and has type strictly smaller than that of $(\A,\B)$.
\end{lemma}
\begin{proof}
By assumption, there exists a Hardy field $\H$, functions $a,b\in
\cG\cap \H$, and $a_1,\ldots, a_m\in \cF(a)$, $b_1,\ldots, b_m \in
\cF(b)$, such that $\A=(a_1,\ldots,a_m)$, $\B=(b_1,\ldots,b_m)$.
Given a pair of functions $(\tilde{a},\tilde{b})\in (\A,\B)$ and
$h\in\N$, the family $(\tilde{a},\tilde{b},h)\vdc(\A,\B)$ is an
ordered family of pairs of functions, all of them of the form
$$
(S_ha_i-\tilde{a},S_hb_i-\tilde{b}), \  \text{ or } \
(a_i-\tilde{a}, b_i-\tilde{b}).
$$

\subsubsection*{We choose $(\tilde{a},\tilde{b})$ as follows:}
  If  the family $\cB'$, defined by \eqref{E:B'}, is non-empty, then we take $\tilde{a}=0$ and let
  $\tilde{b}$ be a function in $\cB'$ with minimal degree.
  Then the first row of the matrix type remains unchanged, and
 one easily checks  using Lemma~\ref{L:KeyHardy} in the positive degree case and   Lemma~\ref{L:properties}
 in the $0$ degree case,
 that the
 second row of the matrix type gets ``reduced", leading to a smaller matrix type
 for every $h\in \N$.
Suppose now that the family  $\cB'$ is empty, in which case all the
functions in the family $\cA$ are unbounded.
   If $\A$
  consists of a single function $a_1$, then we choose
  $(\tilde{a},\tilde{b})\=(a_1,b_1)$ and the result follows. Therefore, we can assume
  that $\A$ contains a function other than $a_1$. We consider two
  cases. If $a_i \cong a_1$ for $i=2,\ldots, m$, then we choose
  $(\tilde{a},\tilde{b})\=(a_1,b_1)$.
Otherwise, we choose
  $(\tilde{a},\tilde{b})\in (\A,\B)$ such that
  $\tilde{a}\ncong a_1$  and $\tilde{a}$  be a function in $\A'$ (see \eqref{E:A'}), with minimal degree
  (such a choice exists since $a_1$ has the highest degree in $\A$).

In all cases, for every $h\in \N$,  one checks using Lemmas~
\ref{L:properties} and \ref{L:KeyHardy} that the first row of the
matrix type of $(\tilde{a},\tilde{b},h)\vdc(\A,\B)$ is ``smaller''
than that of $(\A,\B)$, and as a consequence the new family has
strictly smaller type.

It remains to verify that for every  large enough   $h\in\N$   the
ordered family of pairs of functions
$(\tilde{a},\tilde{b},h)\vdc(\A,\B)$ is nice. We remark that,  by
construction,  the first  pair of functions in this family is
$(S_ha_1-\tilde{a}, S_hb_1-\tilde{b})$.

\begin{claim*}
Property~\eqref{it:nice1} of Definition~\ref{D:good} holds for all
large enough $h\in \N$.
\end{claim*}
 To prove the first part of Property~\eqref{it:nice1} it suffices to show
 that for all large enough $h\in \N$
$$ S_ha_1-S_ha_i \to \infty \text{ for } i=2,\ldots, m
$$
 and
$$
S_h a_1 -a_i \to \infty \text{ for } =1,\ldots, m.
$$
The first  property follows immediately from our assumption
$a_1-a_i\to \infty$ for $i=1,\ldots, m$, and  the second property
follows  upon observing that for all large enough $h\in \N$ we have
by Lemma~\ref{L:KeyHardy} that  $a_1/t \ll S_h a_1-a_i$ and our
assumption $\deg (a_1)\geq 1$ which combined with the property
$a_1\in \G$ gives that $t \prec a_1$.

%% To prove the second
%%property  we consider two cases. If $a\ncong a_1$, then for all but
%%one $h\in \N$ the function $S_h a_1-a$ has the same growth rate as
%%the function $a_1$ and by assumption $a_1(t)\to \infty$. On the
%%other hand, if $a\cong a_1$, then $a_i\cong a_1$ for $i=1,\ldots, m$
%%and for all but one $h$ we have $S_h a_1-a_i\cong a_1/t$. Since
%%$\deg (a_1)\geq 2$ we have $a_1(t)/t\to \infty$.

To prove the second  part of Property~\eqref{it:nice1} it suffices
to show that for all large enough $h\in \N$
$$
S_h a_i-\tilde{a}\ll S_h a_1-\tilde{a},  \text{ for } i=1,\ldots, m
$$
and
 $$
 a_i-\tilde{a}\ll S_h a_1-\tilde{a}, \text{ for } i=1,\ldots, m.
 $$
We only prove the first property, the second can be proved in a
similar fashion. We consider two cases. If $\tilde{a}\ncong a_1$,
then by Lemma~\ref{L:KeyHardy} for all but one $h\in \N$ we have
$S_h a_1-\tilde{a}\sim a_1$, and the estimate follows  by our
assumption $a_i\ll a_1$ for $i=1,\ldots, m$. If $\tilde{a}\cong
a_1$, then by construction $\tilde{a}\cong a_i$ for $i=1,\ldots, m$.
Therefore, for all large enough $h\in \N$ we have by
Lemma~\ref{L:KeyHardy} that $S_h a_i-\tilde{a} \sim a_1/t$ for
$i=1,\ldots, m$. The result follows.

%%It remains to show that
%%  $$
%%  (S_ha_1(t)-a(t))/(S_ha_i(t)-a(t))\to
%%1 \text{ implies that }  S_ha_1-a \cong S_ha_i-a,
%%$$ and
%%$$
%%(S_ha_1(t)-a(t))/(a_i(t)-a(t))\to 1 \text{ implies that } S_ha_1-a
%%\cong a_i-a. $$

%%We only prove the first property, the second can be proved in a
%%similar fashion.To prove the first property notice that our
%%assumption implies that for fixed $h$ and $i$ we have
%% $ S_ha_1-S_ha_i \prec S_h a_i-a\ll a_1$. Hence, $a_i/a_1\to 1$ and
%%%% stated claim.

\begin{claim*} Property~\eqref{it:nice2} of Definition~\ref{D:good} holds for all large enough   $h\in \N$. \end{claim*}
It suffices to show that  for all large enough $h\in \N$
$$
S_h b_i-\tilde{b}\prec S_h a_1-\tilde{a},  \text{ for } i=1,\ldots,
m
$$
and
$$
b_i-\tilde{b}\prec S_h a_1-\tilde{a}, \text{  for } i=1,\ldots, m.
$$
 We only prove the first property, the second one can be proved in a similar fashion.
 We consider two cases.

 If $\tilde{a} \ncong a_1$, then by Lemma~\ref{L:KeyHardy}
for all but one  $h\in \N$ we have $S_h a_1-\tilde{a}\sim a_1$, and
so the result follows since by assumption $b_i\prec a_1$ for
$i=1,\ldots, m$.

If $\tilde{a}\cong a_1$, then by construction
$(\tilde{a},\tilde{b})=(a_1,b_1)$ and $a\cong a_i$ for $i=1,\ldots,
m$. It therefore remains to show that for all large enough $h\in \N$
we have  $S_h b_i-b_1\prec S_h a_1-a_1$ for $i=1,\ldots, m$. To see
this, we express  $S_h b_i-b_1$ as $(S_hb_i-b_i) +(b_i-b_1)$.  If
$1\prec b_i$, then $b_i\in \cG$ (by Lemma~\ref{L:Key1}) and
Lemma~\ref{L:properties} gives that for every $h\in \N$ we have
$S_hb_i-b_i \ll b_i/t\prec  a_1/t$. If $b_i\ll 1$, then since
$t\prec a_1$ we still get $S_hb_i-b_i \prec  a_1/t$. Furthermore,
for $i=2,\ldots,m$, by assumption we have $b_i-b_1\prec a_i-a_1$ and
by Lemma~\ref{L:KeyHardy} we have $a_i-a_1\ll a_1/t$. Combining the
above we get for every $h\in \N$ that  $S_h b_i-b_1\prec a_1/t$ for
$i=1,\ldots, m$.
 Since by Lemma~\ref{L:properties} for every $h\in \N$ we have  $S_h a_1-a_1 \sim a_1/t$, the result
 follows.

\begin{claim*}
Property~\eqref{it:nice3} of Definition~\ref{D:good} holds for all
large enough $h$.
\end{claim*}
Equivalently, we claim that for all large enough $h\in \N$
 $$
 S_hb_1-S_hb_i \prec  S_ha_1-S_ha_i,
  \text{ for } i=2,\ldots,m,
    $$
    and
 $$
 S_hb_1-b_i \prec S_ha_1-a_i,
  \text{ for } i=1,\ldots,m.
    $$
    The first property follows immediately from our hypothesis $b_1-b_i \prec a_1-a_i$
   for  $i=2,\ldots,m$. We  verify the second property.   If $a_i\ncong a_1$, then
  by Lemma~\ref{L:KeyHardy} we have for all large enough $h\in \N$ that
     $S_ha_1-a_i\sim a_1$ for $i=2,\ldots,m$. The desired estimate now follows since  by hypothesis
$b_i\prec a_1$
      for $i=1,\ldots,m$. Suppose now that  $a_i\cong a_1$. Then
     Lemma~\ref{L:KeyHardy}   gives  for all
     large enough $h\in \N$ that  $S_ha_1-a_i\sim a_1/t$.
       So it remains to verify that for every
     large enough $h\in \N$ we have $S_hb_1-b_i \prec a_1/t $. To
see this we express  $S_hb_1-b_i$ as $(S_hb_1-b_1)+(b_1-b_i)$.
%%If $1\prec b_i $, then
Our assumptions and Lemma~\ref{L:properties} give that
$S_hb_1-b_1\sim b_1/t\prec a_1/t$ for all $h\in \N$. Furthermore,
our  assumptions and Lemma~\ref{L:KeyHardy}   give that
$b_1-b_i\prec a_1-a_i \ll a_1/t$. Hence, for every $h\in \N$ we have
$S_hb_1-b_i\prec a_1/t$, as desired.
%%It remains to deal with the
%%case were $b_1\ll 1$. Then $S_hb_1-b_1\prec 1\prec a_1/t$ (we used
%%here that $t\prec a_1$) and as before $b_1-b_i\prec a_1/t$.
This completes the proof.
\end{proof}

\subsection{Some ergodic estimates}
We gather here some simple ergodic estimates that will be used in
the proof of   Proposition~\ref{P:CharB2}.

%%\begin{lemma}\label{L:N-VDC}
%%Let  $\{v_n\}_{n\in\N}$ be a bounded  sequence of vectors in a
%%Hilbert space, and $(\Phi_N)_{N\in\N}$ be a F{\o}lner sequence of
%%subsets of $\N$.
%% For every $h\in \N$ we set
%%$$
%%b_h=\overline{\lim}_{N\to\infty}\Big|
%%\frac{1}{|\Phi_N|}\sum_{n\in\Phi_N}<v_{n+h},v_{n}>\Big|.
%%$$

%%Then
%%$$
%%\overline{\lim}_{N\to\infty}
%%\norm{\frac{1}{|\Phi_N|}\sum_{n\in\Phi_N}v_n}^2\leq 4 \ \!
%%\overline{\lim}_{H\to\infty}\frac{1}{H}\sum_{ h=1}^H b_h.
%%$$
%%\end{lemma}

 Using successive applications of Lemma~\ref{L:N-VDC} one can
show the following (see for example Case $1$ of Proposition~5.3 in
\cite{Fr10}):
\begin{lemma}\label{L:linear}
Let  $(X,\X,\mu, T)$ be a system, $f_1,\ldots,f_m\in L^\infty(\mu)$
be functions  bounded by $1$,
 and  $\alpha_1,\ldots,\alpha_m$ be non-zero integers such that  $\alpha_1\neq \alpha_i$ for $i=2,\ldots m$.
 Then there exists $C=C_{m,
\alpha_2,\ldots, \alpha_m}$ such that
$$
\limsup_{N-M\to\infty} \sup_{\norm{f_2}_\infty, \ldots,
\norm{f_m}_\infty\leq 1}\norm{ \frac{1}{N-M}\sum_{n=M}^N
\prod_{i=1}^m T^{[\alpha_i n]}f_i}_{L^2(\mu)} \leq C \ \!
\nnorm{f_1}_{2m,T}.$$
\end{lemma}
The next two lemmas will help us handle  bounded error terms that
later on appear on the iterates of the transformations involved.
%%In subsequent arguments, some bounded error terms that appear on the
%%iterates of the transformations involved. The next two lemmas will
%%help us neutralize their effect.
\begin{lemma}\label{L:ProdTrick}
Let $(X,\X,\mu,T_1,\ldots,T_\ell)$ be a system, $f_1,\ldots,f_m\in
L^\infty(\mu)$ be functions, and  for $i=1,\ldots, m$, $j=1,\ldots,
\ell$, let $(a_{i,j}(n))$ be sequences with integer values. Then for
every $N\in \N$
\begin{equation}\label{E:tre}
\sup_{E\subset \N}\norm{\frac{1}{N}\sum_{n=1}^N
\prod_{i=1}^m(T_1^{a_{i,1}(n)}\cdots T_\ell^{a_{i,\ell}(n)})f_i
\cdot {\bf 1}_E(n)}_{L^2(\mu)}^2 \leq \norm{\frac{1}{N}\sum_{n=1}^N
\prod_{i=1}^m(\tilde{T}_1^{a_{i,1}(n)}\cdots
\tilde{T}_\ell^{a_{i,\ell}(n)})\tilde{f}_i}_{L^2(\tilde{\mu})}
\end{equation}
where $\tilde{T}\=T\times T$, $\tilde{\mu}\=\mu \times \mu$, and
$\tilde{f}\=f\otimes \bar{f}$.
\end{lemma}
\begin{proof}
Letting
$$
F_n\=\prod_{i=1}^m(T_1^{a_{i,1}(n)}\cdots
T_\ell^{a_{i,\ell}(n)})f_i,
$$
we see that the left hand side in \eqref{E:tre} is bounded by
$$
\frac{1}{N^2} \sum_{1\leq m, n\leq N} \Big|\int F_n \cdot \bar{F}_m
\ d\mu\Big|.
$$
It follows that the square of the left hand side in \eqref{E:tre} is
bounded by
$$
\frac{1}{N^2} \sum_{1\leq m, n\leq N} \Big|\int F_n \cdot \bar{F}_m
\ d\mu \Big|^2=\frac{1}{N^2}\sum_{1\leq m, n\leq N} \int G_n \cdot
\bar{G}_m \ d\tilde{\mu} =\norm{\frac{1}{N}\sum_{n=1}^N
G_n}_{L^2(\tilde{\mu})}^2
$$
where
$$
G_n\=\prod_{i=1}^m(\tilde{T}_1^{a_{i,1}(n)}\cdots
\tilde{T}_\ell^{a_{i,\ell}(n)})\tilde{f}_i.
$$
This completes the proof.
\end{proof}
We deduce from this the following:
\begin{lemma}\label{L:ProdTrick'}
Let $(X,\X,\mu,T_1,\ldots,T_\ell)$ be a system, $f_1,\ldots,f_m\in
L^\infty(\mu)$ be functions, and  for $i=1,\ldots, m$, $j=1,\ldots,
\ell$, let $(a_{i,j}(n))$ be sequences with integer values and
$(e_{i,j}(n))$ be sequences that take values in some finite set of
integers $F$. Then for every $N\in \N$
\begin{multline*}
\sup_{E\subset \N}\norm{\frac{1}{N}\sum_{n=1}^N
\prod_{i=1}^m(T_1^{a_{i,1}(n)+e_{i,1}(n)}\cdots
T_\ell^{a_{i,\ell}(n)+e_{i,\ell}(n)})f_i \cdot {\bf
1}_E(n)}_{L^2(\mu)}^2 \leq\\ |F|^{2\ell m } \cdot \max_{c_{i,j}\in
F}\norm{\frac{1}{N}\sum_{n=1}^N
\prod_{i=1}^m(\tilde{T}_1^{a_{i,1}(n)+c_{i,1}}\cdots
\tilde{T}_\ell^{a_{i,\ell}(n)+c_{i,\ell}})\tilde{f}_i}_{L^2(\tilde{\mu})}
\end{multline*}
where $\tilde{T}\=T\times T$, $\tilde{\mu}\=\mu \times \mu$, and
$\tilde{f}\=f\otimes \bar{f}$.
\end{lemma}
\begin{proof}
The $L^2(\mu)$ norm on left hand side is less than
$$
\sum_{j=1}^t \norm{\frac{1}{N}\sum_{n=1}^N
\prod_{i=1}^m(T_1^{a_{i,1}(n)+e_{i,1}(n)}\cdots
T_\ell^{a_{i,\ell}(n)+e_{i,\ell}(n)})f_i \cdot {\bf
1}_{E_j}(n)}_{L^2(\mu)} $$ where the sets $E_1,\ldots, E_t$ ($t\leq
|F|^{\ell m}$) form a partition of $E$ into sets where the sequences
$e_{i,j}$ are all constant.
%%To complete the proof, it suffices to
%%notice that this expression is bounded by
%%$$
%%t \cdot \max_{i=1,\ldots, t} \norm{\frac{1}{N}\sum_{n=1}^N
%%\prod_{i=1}^m(T_1^{a_{i,1}(n)+e_{i,1}(n)}\cdots
%%T_\ell^{a_{i,\ell}(n)+e_{i,\ell}(n)})f_i \cdot {\bf
%%1}_{E_i}(n)}_{L^2(\mu)} $$  and use
 The desired estimate is now an immediate consequence of Lemma~\ref{L:ProdTrick}.
\end{proof}

\subsection{Proof of Proposition~\ref{P:CharB2}}
We start with  an elementary  lemma that will be used to prove
seminorm estimates
   in the case where all the iterates have sub-linear growth.
\begin{lemma}\label{L:ChangeVar} Let $a\colon [c,\infty)\to \R$ be a positive Hardy field function that
satisfies the growth condition  $\log{t} \prec a(t) \prec t$ and
$(A(n))$ be a bounded sequence in a normed space such that $
\lim_{N-M\to\infty}\norm{\frac{1}{N-M}\sum_{n=M}^{N} A(n)}=0 $. Then
$ \lim_{N\to\infty}\norm{\frac{1}{N}\sum_{n=1}^N A([a(n)])}=0.$
\end{lemma}
\begin{remark}
When $t^\varepsilon \prec a(t)\prec 1$ for some $\varepsilon>0$, the
conclusion holds under the weaker assumption $
\lim_{N\to\infty}\norm{\frac{1}{N}\sum_{n=1}^{N} A(n)}=0 $.
\end{remark}
\begin{proof}
Letting $w(n)=\{k\in \N\colon [a(k)]=n\}$ and $W(N)=w(1)+\cdots
+w(N)$, it suffices to show that $
\lim_{N\to\infty}\norm{\frac{1}{W(N)}\sum_{n=1}^N w(n)\cdot
A(n)}=0.$ Letting $b(t)=a^{-1}(t)$, one  checks that $w(n)/
(b(n+1)-b(n))\to 1$ and $W(n)/b(n)\to 1$. Our assumptions give that
$\log(b(t))\prec t\prec b(t)$. This implies that $b(t+1)-b(t)\to
\infty$ and $(b(t+1)-b(t))/b(t)\to 0$. Hence, $w(n)\to \infty$ and
$w(n)/W(n)\to 0$. The needed convergence to $0$ now follows from
Theorem~3.6 in \cite{BH09}.
\end{proof}

 We are now in position  to
prove Proposition~\ref{P:CharB2}. Given  a Hardy field $\H$ and
functions  $a,b\in \cG\cap \H$ our goal is to establish the
following claim:

\medskip\noindent\textbf{Claim:} Let  $a_i\in \cF(a),$
   $b_i \in \cF(b) $  for $i=1,\ldots,m$, and   $(\A,\B)$ be a nice family of
ordered pairs of functions where $\A\=(a_1,\ldots,a_m)$,
$\B\=(b_1,\ldots,b_m)$. Let $W$ be the
  matrix type of  this family.
 Then there exists $k=k(W,m)\in \N$ such that: If
$\nnorm{f_1}_{k,T_1}=0$, then  the averages
\begin{equation}\label{E:fgh'}
\frac{1}{N}\sum_{n=1}^N \prod_{i=1}^m (T_1^{[a_i(n)]}
T_2^{[b_i(n)]})f_i
\end{equation}
converge to $0$  in $L^2(\mu)$.

Note that the conclusion of Proposition~\ref{P:CharB2} is somewhat
stronger in two respects: $(i)$ The integer $k$ depends  only on the
degree of the family. This strengthening easily follows from the
above mentioned claim after noticing that there is only a finite
number of possible matrix types for families that have fixed degree
and numbers of pairs of functions. $(ii)$ The conclusion  involves a
supremum over
 all subsets of $\N$.
 This strengthening follows  by combining the above mentioned
 statement with
Lemma~\ref{L:ProdTrick} and the fact that $\nnorm{f}_{k+1,T}=0$
implies that $\nnorm{f\otimes \bar{f}}_{k,T\times T}=0$ (this
follows from   \eqref{E:seminonergodic}).

We  proceed to prove the  claim  by induction on the type of the
nice family  $(\A,\B)$.

\medskip\noindent\textbf{Base Case:}
Suppose that  $\deg (a_1)=0$, in which case, for $i=1,\ldots, m$ the
functions $a_i$ and $b_i$   have sub-linear growth. We are going to
show that if $\nnorm{f_1}_{2m+1,T_1}=0$, then the averages
\eqref{E:fgh'} converge to $0$ in $L^2(\mu)$.

Our assumption implies that  for $i=2,\ldots, m$ one has
$$
a_i(t)=\alpha_i a_1(t)+c_i(t)
$$ for some $\alpha_i\in \R$ and functions $c_i$ that satisfy $c_i\prec a_1$. It is important to note that $\alpha_i\neq
1$ for $i=2,\ldots m$. Otherwise $a_1-a_i\prec a_1$, and since
$a_1-a_i\in \cF(a)$ and $\deg (a_1)=0$, we deduce by
Lemma~\ref{L:Key1} that $a_1-a_i\to 0$, contradicting our assumption
that the family $(\A,\B)$ is nice.
 Let
$$
 \tilde{b}_i\=b_i\circ a_1^{-1}, \quad \tilde{c}_i\=c_i\circ a_1^{-1}.
 $$
 (We caution the reader that these functions are not necessarily
 Hardy field functions.)
 Since $b_i\prec a_1$ and $c_i\prec a_1$ we have  $\tilde{b}_i\prec 1$ and  $\tilde{c}_i\prec 1$.
Furthermore, one sees that
$$
[a_i(n)]=[\alpha_i[a_1(n)]] +[\tilde{c}_i([a_1(n)])]+e_{i}(n),\quad
[b_i(n)]=[\tilde{b}_i([a_1(n)])]+e'_{i}(n),$$ where the sequences
$(e_{i}(n))$, $(e'_{i}(n))$ take finitely many integer values.
Therefore, it suffices to show that the averages in $n$ of
$$
(T_1^{[a_1(n)]} T_2^{[\tilde{b}_1([a_1(n)])] +e'_1(n)})f_1\cdot
\prod_{i=2}^m (T_1^{[\alpha_i[a_1(n)]]
+[\tilde{c}_i([a_1(n)])]+e_i(n)} T_2^{[\tilde{b}_i([a_1(n)])]
+e'_i(n)})f_i
$$
converge to $0$  in $L^2(\mu)$.

By Lemma~\ref{L:ProdTrick'} it suffices to show that  the averages
in $n$ of
$$
(\tilde{T}_1^{[a_1(n)]}
\tilde{T}_2^{[\tilde{b}_1([a_1(n)])]})\tilde{f}_1\cdot \prod_{i=2}^m
(\tilde{T}_1^{[\alpha_i[a_1(n)]] +[\tilde{c}_i([a_1(n)])]}
\tilde{T}_2^{[\tilde{b}_i([a_1(n)])]})\tilde{f}_i
$$
converge to $0$  in $L^2(\tilde{\mu})$ for all $\tilde{f}_i\in
L^\infty(\tilde{\mu}), i=2,\ldots, m$, where $\tilde{T}\=T\times T$,
 $\tilde{\mu}\=\mu \times \mu$, and
$\tilde{f}\=f\otimes \bar{f}$. Using Lemma~\ref{L:ChangeVar} we can
further reduce matters to showing that   for every   sequence
$(I_N)$ of intervals of integers with lengths increasing to
infinity, the averages
$$
\frac{1}{|I_N|}\sum_{n\in
I_N}(\tilde{T}_1^n\tilde{T}_2^{[\tilde{b}_1(n)]}) \tilde{f}_1\cdot
\prod_{i=2}^m(\tilde{T}_1^{[\alpha_i
n]+[\tilde{c}_i(n)]}\tilde{T}_2^{[\tilde{b}_i(n)]}) \tilde{f}_i
$$
converge to $0$  in $L^2(\tilde{\mu})$ as $N\to\infty$.

Using our assumptions, one easily sees  that the functions
$\tilde{c}_i(t+1)-\tilde{c}_i(t)$ and
$\tilde{b}_i(t+1)-\tilde{b}_i(t)$ converge to $0$ and have
eventually constant sign. Because of this,  it is possible to
decompose each interval $I_N$ (except a finite set with fixed
cardinality) into sub-intervals with length tending to infinity, and
such that for every $N\in \N$ the sequences
$([\tilde{c}_2(n)]),\ldots, ([\tilde{c}_m(n)])$ and
$([\tilde{b}_1(n)]), \ldots, ([\tilde{b}_m(n)])$ are constant on
each interval. Thus, without loss of generality we can assume that
all these sequences are constant in each interval $I_N$.  Then the
desired fact would follow if we prove that the averages
$$
\frac{1}{|I_N|}\sum_{n\in I_N}\tilde{T}_1^n  \tilde{f}_1\cdot
\prod_{i=2}^m (\tilde{T}_1^{[\alpha_i n]
+c_{i,N}}\tilde{T}_2^{d_{i,N}}) \tilde{f}_i
$$
converge to $0$ in $L^2(\tilde{\mu})$ as $N\to \infty$, for every
choice of integers  $c_{i,N}, d_{i,N}$. This follows form
Lemma~\ref{L:linear} and the fact that $\nnorm{f_i}_{2m+1,T_i}=0$
implies that $\nnorm{\tilde{f}_i}_{2m,\tilde{T_i}}=0$.

\medskip\noindent\textbf{Inductive step:}
Let  now $(\A,\B)$ be a nice family of $m$ ordered pairs of
functions, of matrix type $W$,  and such that $\deg (a_1)\geq 1$.
Suppose that the statement we want  to prove holds for every nice
family of $2m$ ordered pairs of functions with  matrix type $W'$
strictly less than $W$ (there is a finite number of such families),
and let $k(W',2m)$ be the integer for which the conclusion of the
corresponding statement holds. We let
$k(W,m)=\max_{W'<W}(k(W',2m))+1$. Our goal is to  show that $k(W,m)$
works for the family $(\A,\B)$. Since in the  base case we covered
all nice families with degree $0$, this is going to complete the
induction.

So assuming that $\nnorm{f_1}_{k(W,m),T_1}=0$, we want to show that
the averages \eqref{E:fgh'} converge to $0$ in $L^2(\mu)$. By
Lemma~\ref{L:N-VDC} it suffices to show that for large enough $h\in
\N$ the averages in $n$ of
$$
\int \prod_{i=1}^m (T_1^{[a_i(n+h)]} T_2^{[b_i(n+h)]})f_i \cdot
 (T_1^{[a_i(n)]} T_2^{[b_i(n)]})\bar{f}_i \ d\mu
$$
converge to $0$. We compose with
$T_1^{-[\tilde{a}(n)]}T_2^{-[\tilde{b}(n)]}$, where
$(\tilde{a},\tilde{b})\in (\P,\Q)$ is chosen as in
Lemma~\ref{L:reduceA}, and use the Cauchy-Schwarz inequality.
  This  reduces matters to showing that  for every   large enough
$h\in \N$  the averages in $n$ of
$$
\prod_{i=1}^m (T_1^{[a_i(n+h)-\tilde{a}(n)]+e_{1,i}(n)}
T_2^{[b_i(n+h)-\tilde{b}(n)]+e_{2,i}(n)})f_i \cdot
(T_1^{[a_i(n)-\tilde{a}(n)]+e_{3,i}(n)}
T_2^{[b_i(n)-\tilde{b}(n)]+e_{4,i}(n)})\bar{f}_i
$$
 converge to $0$ in $L^2(\mu)$ where $e_{i,j}$ are sequences that take values in the set $\{0,1\}$.
 By Lemma~\ref{L:ProdTrick'} it
 suffices to show that the averages in $n$ of
 \begin{equation}\label{E:pop}
\prod_{i=1}^m (\tilde{T}_1^{[a_i(n+h)-\tilde{a}(n)]+c_{1,i}}
\tilde{T}_2^{[b_i(n+h)-\tilde{b}(n)]+c_{2,i}})\tilde{f}_i \cdot
(\tilde{T}_1^{[a_i(n)-\tilde{a}(n)]+c_{3,i}}
\tilde{T}_2^{[b_i(n)-\tilde{b}(n)]+c_{4,i}})\bar{\tilde{f}}_i
\end{equation}
 converge to $0$ in $L^2(\tilde{\mu})$, where, $c_{i,j}$ are constants
 with values either $0$ or $1$,
  $c_{1,1}=c_{2,1}=0$,
  and  $\tilde{T}\=T\times T$,
 $\tilde{\mu}\=\mu \times \mu$,
$\tilde{f}\=f\otimes \bar{f}$.
 We remove the functions that happen
to be composed with eventually constant iterates of  $T_1$ and $T_2$
(this will happen when the functions involved are bounded), since
they do not affect convergence to $0$. This corresponds to the
operation $^\ast$ defined in Section~\ref{subsec:vdcoperation}, and
the resulting multiple ergodic averages  are associated with  the
families of functions
 $(\tilde{a},\tilde{b},h)\vdc(\A,\B)$. Our final   goal is to show that these averages
 convergence to $0$ in $L^2(\mu)$ for every large enough  $h\in\N$.

 By Lemma~\ref{L:reduceA}, for every large enough  $h\in\N$,  the
family $(\tilde{a},\tilde{b},h)\vdc(\A,\B)$ is nice, has type $W'$
strictly smaller than $W$, and its first pair is
 $([a_1(n+h)-\tilde{a}(n)],[b_1(n+h)-\tilde{b}(n)])$.
Notice also that in \eqref{E:pop}  the iterate
$T_1^{[a_1(n+h)-\tilde{a}(n)]}T_2^{[b_1(n+h)-\tilde{b}(n)]}$ is
applied to the function $\tilde{f}_1$. Since
$\nnorm{f_1}_{k(W,m),T_1}=0$ implies that
$\nnorm{\tilde{f}_1}_{k(W',2m),\tilde{T}_1}=0$, the induction
hypothesis applies and proves convergence to $0$ in $L^2(\mu)$.
 This completes the proof of Proposition~\ref{P:CharB2}.

\section{Seminorm estimates for the
highest degree iterate:  The general case}\label{S:Char}

The next proposition is the generalization of
Proposition~\ref{P:CharB2} to the case of an arbitrary number of
transformations. To avoid unnecessary repetition, we
 define the concepts needed in the proof of  Proposition~\ref{P:CharB},
and then only summarize its proof providing details only when
non-trivial modifications of the arguments used in the previous
section are needed.

\begin{proposition}\label{P:CharB}
%%Let $\ell,m\in\N$,
%%Let $\ell\geq 2$ and $m\geq 1$ be integers,
 Let $(X,\X,\mu,T_1,\ldots,T_\ell)$ be a system,
and $f_1,\ldots,f_m\in L^\infty(\mu)$. Suppose that
$(\A_1,\ldots,\A_\ell)$  is a nice  ordered  family of $\ell$-tuples
of functions with degree $d$ (all notions are defined below). Then
there exists $k=k(d,\ell,m)\in \N$   such that: If
$\nnorm{f_1}_{k,T_1}=0$, then
$$
\lim_{N\to\infty} \sup_{E\subset \N} \norm{
\frac{1}{N}\sum_{n=1}^{N}
 \prod_{i=1}^m (T_1^{[a_{1,i}(n)]}\cdots
 T_\ell^{[a_{\ell,i}(n)]})f_i \cdot {\bf 1}_E(n)}_{L^2(\mu)}=0.
$$
\end{proposition}
Applying this result to the nice family $(\A_1,\ldots,\A_\ell)$
where $\A_1\=(a_1,0,\ldots,0)$, $\A_2\=(0,a_2,\ldots,0)$, ...
$\A_\ell\=(0,\ldots,0,a_\ell)$,  we get:
\begin{proposition}
\label{P:CharBspecial}
 Let $(X,\X,\mu,T_1,\ldots,T_\ell)$ be a system,
and $f_1,\ldots,f_\ell\in L^\infty(\mu)$ be functions. Let $\H$ be a
Hardy field and  $a_1,\ldots,a_\ell\in \cG\cap \H$ be functions with
different growth and highest degree $d\=\deg(a_1)$. Then there
exists $k=k(d,\ell)$ such that: If $\nnorm{f_1}_{k,T_1}=0$, then the
averages
$$
\frac{1}{N}\sum_{n=1}^{N} \prod_{i=1}^\ell T_i^{[a_i(n)]}f_i
$$
converge to $0$ in $L^2(\mu)$.
\end{proposition}

\subsection{Families of $\ell$-tuples and their types}
\label{SS:4.1}
\subsubsection{Families of $\ell$-tuples of functions}
Let $\ell,m\in\N$. Given $\ell$ ordered families of functions
$$
\A_1\=(a_{1,1},\ldots,a_{1,m}) ,\ldots,
\A_\ell\=(a_{\ell,1},\ldots,a_{\ell,m})
$$
we define an \emph{ordered family of $\ell$-tuples of functions} as
follows
$$
(\A_1,\ldots,\A_\ell)\=\big((a_{1,1},\ldots,a_{\ell,1}),\ldots,(a_{1,m},\ldots,a_{\ell,m})\big).
$$
The maximum of the degrees of the functions in the families
$\A_1,\ldots,\A_\ell$ is called \emph{the degree of  the family}
$(\A_1,\ldots,\A_\ell)$.

For convenience of exposition, if $\ell$-tuples of bounded functions
appear in $(\A_1,\ldots,\A_\ell)$ we remove them, and henceforth we
assume:
%%\begin{itemize}\em
%%\item[]

\emph{All families  $(\A_1,\ldots, \A_\ell)$  that we consider do
not contain $\ell$-tuples of bounded functions.}
%%\end{itemize}

\subsubsection{Definition of type}
We fix $d\geq 0$ and restrict ourselves to families of degree
between $0$ and $d$. We define
$$
\A_1'\=\{  a_{1,j} \in \A_1\colon a_{1,j} \text{ is not bounded }\}
$$
and for $i=2,\ldots, \ell$
$$
\A_i'\=\{a_{i,j} \in \A_i\colon a_{i,j} \text{ is not bounded and }
 a_{i',j} \text{ is
bounded for } i'<i\}.
$$

For $i=1,\ldots, \ell$ and $j=0,1,\ldots, d$, we  let $w_{i,j}$ be
the number of distinct non-equivalent classes of  functions  of
degree $j$ in the family  $\A_i'$. We define the \emph{(matrix)
type} of the family $(\A_1,\ldots, \A_\ell)$ to be the matrix
$$
\begin{pmatrix}
w_{1,d}& \ldots &  w_{1,0}\\ w_{2,d}& \ldots& w_{2,0}\\ \vdots & \ldots &\vdots \\
w_{\ell,d}& \ldots& w_{\ell,0}
\end{pmatrix}.
$$

%% For example, let $d=4$,
%% and consider the family of triples of functions
%%\begin{multline*}
%%\big((t^{3.5}, t^{4},t^{4.5}), \ (t^{2.5}+t,3t^{3.5},0), \ (t^{2.5}, 0,2t), \ (t^{0.5},t^{1.5},0), \\
%%\ (0,t^{2.5},t^{4.5}), \  (0,2t^{2.5}, n^{3.5}), \ (0,0,t^{3.5}), \
%%(0,0,t^{3.1})\big).
%%\end{multline*}
%% Since
%%$$
%% \A_1'=\{t^{3.5},t^{2.5}+t,t^{2.5},t^{0.5}\},\quad   \A_2'=\{t^{2.5},2t^{2.5}\}, \quad
%% \A_3'=\{ t^{3.5},t^{3.1} \},
%%$$
%%the   type of this family is
%%$$
%%\begin{pmatrix}
%%0& 1& 2& 0& 1\\ 0& 0& 2& 0 & 0\\0& 2& 0&0 &0
%%\end{pmatrix}.
%%$$

As in Section~\ref{SS:type}, we order these  types
lexicographically.
%%Given two $\ell\times (d+1)$ matrices
%%$W\=(w_{i,j})$ and $W'\=(w'_{i,j})$, we say that the first is bigger
%%than the second, and write $W\succ W'$, if $w_{1,d}>w'_{1,d}$, or
%%$w_{1,d}=w'_{1,d}$ and $w_{1,d-1}>w'_{1,d-1}$, $\ldots$, or
%%$w_{1,j}=w'_{1,j}$ for $j=0,1,\ldots,d$ and $w_{2,d}>w'_{2,d}$, and
%%so on.
 The following extension of Lemma~\ref{lem:decreasing} holds:
\begin{lemma}
\label{lem:decreasing2} Every decreasing sequence of
 types of families of  $\ell$-tuples   is stationary.
\end{lemma}
\subsection{Nice families and the van der Corput operation}
%%We  define a class of families of $\ell$-tuples of functions and an
%%important operation that preserves such families and reduces their
%%type.

\subsubsection{Nice families} Henceforth, we are going to work with
families of $\ell$-tuples of functions that satisfy the following
properties:
\begin{definition}
Let $\H$ be a Hardy field,  $a_1, \ldots, a_\ell \in \cG\cap \H$ be
functions, $a_{i,j}\in \F(a_i)$ for $i=1,\ldots, \ell$, $j=1,
\ldots, m$, and $\A_1\=(a_{1,1},\ldots,a_{1,m})$,$\ldots$,
$\A_\ell\=(a_{\ell,1},\ldots,a_{\ell,m})$. We call the ordered
family $(\A_1,\ldots,\A_\ell)$  of $\ell$-tuples of functions
\emph{nice} if
\begin{enumerate}
\item\label{it:nice21}
 $a_{1,1}- a_{1,j}\succ 1$ and $a_{1,j}\ll a_{1,1}$ for $j=2,\ldots,m$;
\medskip

\item\label{it:nice22}
$a_{i,j}\prec a_{1,1}$ for $i=2,\ldots,\ell$, $j=1,\ldots,m$;

\medskip
\item\label{it:nice23}
$a_{i,1}-a_{i,j} \prec a_{1,1}-a_{1,j} $ for $i=2,\ldots,\ell$,
$j=2,\ldots,m$.
\end{enumerate}
\end{definition}

\subsubsection{The van der Corput operation}
Given a family $\A\=\big(a_1,\ldots,a_m\big)$,    a function
$a\colon [c,\infty) \to \R$, and $h\in\N$, we define $$
 S_h\A\=(S_ha_1,\ldots,S_ha_m) \text{ and }
 \A-a\=\big(a_1-a,\ldots,a_m-a\big).
 $$
  Given a family of
$\ell$-tuples of functions  $(\A_1,\ldots,\A_\ell)$,
 an $\ell$-tuple $(\tilde{a}_1,\dots,\tilde{a}_\ell)\in (\A_1,\dots\A_\ell)$, and $h\in\N$,
we define the following operation
 $$
(\tilde{a}_1,\ldots, \tilde{a}_\ell,h)\vdc(\A_1,\ldots, \A_\ell)
\=(\tilde{\A}_{1,h},\ldots \tilde{\A}_{\ell,h})^*
 $$
 where
 $$
  \tilde{\A}_{i,h}\=(S_h\A_i-\tilde{a}_i,\A_i-\tilde{a}_i).
 $$
for $i=1,\ldots,\ell$,  and  $^\ast$ is the operation that removes
all   $\ell$-tuples that consist of bounded  functions  from a given
family of
 $\ell$-tuples of functions.

\subsection{Reducing the type.}
The next lemma enables us to reduce the type of a nice family of
$\ell$-tuples that has  positive degree:
\begin{lemma}\label{L:reduceA'}
Let $(\A_1,\ldots, \A_\ell)$ be  a nice family of $\ell$-tuples of
functions with $\deg(a_{1,1})\geq 1$.
 Then there exists $(\tilde{a}_1,\ldots,\tilde{a}_\ell)\in (\A_1\cup \{0\},\ldots,\A_\ell\cup \{0\})$
 such that for every large enough $h\in\N$
the family  $(\tilde{a}_1,\ldots,\tilde{a}_\ell,h)\vdc(\A_1,\ldots,
\A_\ell)$ is nice and has strictly smaller type than
  $(\A_1,\ldots, \A_\ell)$.
\end{lemma}
\begin{proof}
  Let
$\A_i\=(a_{i,1},\ldots,a_{i,m})$ for $i=1,\ldots,\ell$.
%%We remind
%%the reader that  the  family
%%$(\tilde{a}_1,\ldots,\tilde{a}_\ell,h)\vdc(\A_1,\ldots, \A_\ell)$
%%consists of $\ell$-tuples of functions that have the form
%%$$
%%(S_ha_{1,j}-\tilde{a}_1,\ldots, S_ha_{\ell,j}-\tilde{a}_\ell), \ j=
%%1,\ldots,m, \ \text{ or } \ (a_{1,j}-\tilde{a}_1, \ldots, a_{\ell,j}
%%-\tilde{a}_\ell).
%%$$
Let $i\in \{1,\ldots, \ell\}$ be the largest integer such that the
family $\A_i'$ is non-empty. We choose $(\tilde{a}_1,\ldots,
\tilde{a}_\ell)$ as follows:
%%  If the family $\A_\ell'$ is non-empty,
%%then we take $\tilde{a}_1=\cdots=\tilde{a}_{\ell-1}=0$ and let
%%$\tilde{a}_\ell$ be a
%% function  in $\A_\ell'$ with minimal degree.
%%Then for every $h\in\N$, the first $\ell-1$ rows of the matrix type
%%remain unchanged, and one easily checks using Lemma~\ref{L:KeyHardy}
%%that the last row will get ``reduced'', leading to a smaller matrix
%%type. Similarly,
If  $i\neq 1$ (in which case $\A_\ell', \A_{\ell-1}',\ldots,
\A_{i+1}'$ are empty, and $\A_{i}'$ is non-empty), then we take
$\tilde{a}_1=\cdots=\tilde{a}_{i-1}=0$ and let $\tilde{a}_i$ to be a
function  of minimal degree in $\A_i'$. Then for every $h\in\N$, one
checks using Lemmas~\ref{L:properties} and \ref{L:KeyHardy} that the
first $i-1$ rows of the matrix type remain unchanged, and the
$i$-the row will get ``reduced'', leading to a smaller matrix type.

If   $i=1$, then  the families  $\A_\ell', \A_{\ell-1}',\ldots,
\A_{2}'$ are all empty. If $\A_1$ consists of a single function,
namely $a_{1,1}$, then we choose
$(\tilde{a}_1,\ldots,\tilde{a}_\ell)\=(a_{1,1},\ldots, a_{\ell,1})$
and the result follows  using Lemma~\ref{L:properties}. Therefore,
we can assume that $\A_1$ contains some function  other than
$a_{1,1}$. We consider two cases. If $a\cong a_{1,1}$ for all $a\in
\A_1$, then we choose
$(\tilde{a}_1,\ldots,\tilde{a}_\ell)\=(a_{1,1},\ldots, a_{\ell,1})$.
Otherwise, we choose $(\tilde{a}_1,\ldots,\tilde{a}_\ell)\in
(\A_1,\ldots,\A_\ell)$ with $\tilde{a}_1\ncong a_{1,1}$, and such
that  $\tilde{a}_1$ is a function in $\A_1'$ with minimal  degree
(such a choice exists since $a_{1,1}$ has the highest degree in
$\A_1$). In all cases, for every $h\in\N$, one checks  using
Lemmas~\ref{L:properties} and \ref{L:KeyHardy}  that  the first row
of the matrix type of
$(\tilde{a}_1,\ldots,\tilde{a}_\ell,h)\vdc(\A_1,\ldots, \A_\ell)$ is
``smaller'' than that of $(\A_1,\ldots, \A_\ell)$.

It remains to verify that for  large enough $h\in\N$ the family
 $(\tilde{a}_1,\ldots,\tilde{a}_\ell,h)\vdc(\A_1,\ldots, \A_\ell)$
 is nice. This   argument is very similar to   the one used in  Lemma~\ref{L:reduceA} and so we omit it.
\end{proof}

  \subsection{Proof of Proposition~\ref{P:CharB}}  Proposition~\ref{P:CharB}
is proved by an induction on  the type of nice families of
$\ell$-tuples of functions.  The base case covers all families with
degree $0$ and is proved in a way completely analogous to the case
$\ell=2$, that was treated in the previous section. The inductive
step is also completely analogous  to the case $\ell=2$ and is
omitted.

\section{Correlation estimates  }
 In order to motivate the estimates that are proved in this section we
 recap part of our plan for  studying the  limiting behavior of the averages
\begin{equation}
\label{E:223} \frac{1}{N}\sum_{n=1}^{N}
 T_1^{[a_1(n)]}f_1 \cdot T_2^{[a_2(n)]}f_2
\end{equation}
when $a_2\prec a_1$. We showed in Proposition~\ref{P:CharB2special}
that there exists $d\in \N$ such that if $\nnorm{f_1}_{d,T_1}=0$,
then the averages \eqref{E:223} converge to $0$ in $L^2(\mu)$. Our
goal is  to prove a similar result for the function $f_2$. Using
the decomposition result of Proposition~\ref{P:ApprDual}
 we can reduce matters to showing
 that there exists $d\in \N$ such that if $\nnorm{f_2}_{d,T_2}=0$,
 then
$$
\frac{1}{N}\sum_{n=1}^{N} \cD_x([a_1(n)])\cdot f_2(T_2^{[a_2(n)]}x)
\to^{L^2(\mu)} 0,
$$
where $(\cD_x(n))$ is a uniformly bounded sequence of  measurable
functions  such that for almost every $x\in X$ the sequence
$(\cD_n(x))$ is a dual sequence of  level at most  $d$. This
motivates us to seek for  estimates that connect   averages of the
form
$$
\frac{1}{N}\sum_{n=1}^{N}     \cD([a(n)])\cdot A(n)
$$
where $(\cD(n))$ is a dual sequence, and averages involving only
product of translates of the sequence $(A(n))$. We produce  such
estimates in this section.

\subsection{Correlation estimates for sequences}
 %%Given  functions $b_1,\ldots,b_l\colon [c,\infty)\to\R$
%%  we define
%%  $$
%%\mathcal{F}(b_1,\ldots, b_l)\=\big\{ \text{integer combinations of
%%functions of the form } S_hb_i, h\in\N, i\in \{1,\ldots, l\}\big\},
%%  $$
%%  and
%%$\mathcal{S}_{b_1,\ldots, b_l}$   be the  algebra of subsets of $\N$
%% spanned by sets of the form $$\{n\colon [f(n)]=k\}, \quad \{n\colon [f(n)]-[g(n)]
%%=[f(n)-g(n)]\}, \quad \{n\colon [f(n)]-[g(n)] =[f(n)-g(n)] +1\}$$
%%where  $f,g\in \mathcal{F}(b_1,\ldots, b_l)$ and $k\in \Z$.
%%Obviously,   the collection  $\mathcal{S}_{b_1,\ldots, b_l}$ is
%%always countable.

%% The essence of the next result is that if a sequence a uniform
%%enough function  does not correlate with any dual sequence composed
%%with a  sequence arising from a Hardy field function.
\begin{proposition}\label{P:BASIC}
Let $\H$ be a Hardy field and $b_1,\ldots, b_l\in \H$ be functions
with  maximum degree $d\geq -1$. Let $(X,\X,\mu)$ be a probability
space, $(A(n))$, $(\cD_{1}(n)), \ldots,(\cD_{l}(n)) $ be uniformly
bounded sequences of $L^\infty(\mu)$ functions, such that for almost
every $x\in X$, for $i=1,\ldots, l$,  the sequences $(\cD_{i,x}(n))$
are dual sequences of level at most $r\in \N$. Then there exists
$s_0= s_0(d,l,r)\in \N$ and $C=C(d,l,r)\in \R$ such that for some
$s\leq s_0$ we have
\begin{multline*}
\limsup_{N\to\infty}\norm{\frac{1}{N}\sum_{n=1}^N (A_x(n)\cdot
\prod_{i=1}^l\cD_{i,x}([b_i(n)]))}_{L^2(\mu)}^{2^s}\leq \\
C\cdot \limsup_{H_s\to\infty}\frac{1}{H_s}\sum_{h_s=1}^{H_s} \cdots
\limsup_{H_1\to\infty}\frac{1}{H_1}\sum_{h_1=1}^{H_1}
\limsup_{N\to\infty} \sup_{E\subset \N}\norm{\frac{1}{N}\sum_{n=1}^N
 \prod_{\e\in \{0,1\}^s}  \cC^{|\e|} A_x(n+\e\cdot \h) \cdot {\bf
 1}_{E}(n)}_{L^2(\mu)}
\end{multline*}
where $\h\=(h_1,\ldots, h_s)$. %%and $\mathcal{S}_{b_1,\ldots,b_l}$ as before.
\end{proposition}
\begin{remark}
Notice that  we do not have to assume that $b_1,\ldots, b_l\in \G$.
When $\ell=1$ and $b_1(t)=t$ the result was proved in \cite{HK09}.
\end{remark}
\begin{proof}
To begin with,  using identity  \eqref{E:alternate}, we see that
there exist $k,\ell \in \N$ (in fact, $k=lr$ and $\ell=l(2^r-1)$),
vector valued sequences of functions $\b_1,\ldots, \b_\ell\colon
[c,\infty) \to\R^k$, with coordinates functions $b_{i,j}$ taken from
the set $\{0, b_1,\ldots, b_l\}$, and sequences $\d_1,\ldots,
\d_\ell\colon \N^k \to L^\infty(\mu)$, such that
$$
\prod_{i=1}^l\cD_{i,x}([b_i(n)]))= \lim_{M\to\infty}
\frac{1}{M^k}\sum_{\m \in [1,M]^k} \ \prod_{i=1}^\ell \d_{i,x}(\m+
[\b_i(n)] )
$$
where $[ \b_i ]\=([b_{i,1}],\ldots, [b_{i,k}])$ and
$\m\=(m_1,\ldots,m_k)$. Furthermore,  all functions  $b_{i,j}$ and
  $\d_i$ are bounded by $1$. It therefore  suffices to
prove the following claim:

\noindent {\bf Claim:} Let  $k, \ell\in \N$, $\H$ be a Hardy field,
and  for $i=1,\ldots, \ell$ let $\b_i=(b_{i,1},\ldots, b_{i,k})$
where $b_{i,j}\in \H$ are functions with maximum degree $d\geq -1$.
Furthermore, let  $(A(n))$, $(\d_1(\m)),\ldots, (\d_\ell(\m))$, $\m
\in \N^k$, be sequences of $L^\infty(\mu)$ functions,  all bounded
by $1$. Then there exists $s_0=s_0(d,k,\ell)\in \N$
%%and subsets $S_\n$ of $\N$, that do not
%%depend on the sequences $(A(n))$,
such that for some $s\leq s_0$ the expression
\begin{equation}\label{E:uty1}
\limsup_{N\to\infty} \sup_{\norm{\d_i}_\infty\leq 1, E\subset \N}
\limsup_{M\to\infty} \norm{ \frac{1}{N}\sum_{n=1}^N \ \big(
A_x(n)\cdot \frac{1}{M^k}\sum_{\m \in [1,M]^k}  \ \prod_{i=1}^\ell
\d_{i,x}(\m+ [\b_i(n)] ) \cdot {\bf 1}_E(n)\big)}_{L^2(\mu)}^{2^s}
\end{equation}
is bounded by a constant $C=C(d,k,\ell)$ times
$$
 \limsup_{H_s\to\infty}\frac{1}{H_s}\sum_{h_s=1}^{H_s}\cdots
\limsup_{H_1\to\infty}\frac{1}{H_1}\sum_{h_1=1}^{H_1}\limsup_{N\to\infty}
  \sup_{E\subset \N} \norm{
  \frac{1}{N}\sum_{n=1}^N
  \prod_{\e\in \{0,1\}^s}
 \cC^{|\e|}A_x(n+\e \cdot \h) \cdot {\bf 1}_{E}(n)}_{L^2(\mu)}
$$
where $[ \b_i ]\=([b_{i,1}],\ldots, [b_{i,k}])$   and
$\h\=(h_1,\ldots, h_s)$.

Equivalently, it suffices to prove the same estimate with the left
hand side replaced with
$$
\limsup_{N\to\infty}\limsup_{M\to\infty} \norm{
\frac{1}{N}\sum_{n=1}^N \ \big( A_x(n)\cdot \frac{1}{M^k}\sum_{\m
\in [1,M]^k}  \ \prod_{i=1}^\ell \d_{i,x,N}(\m+ [\b_i(n)] ) \cdot
{\bf 1}_{E_N}(n)\big)}_{L^2(\mu)}^{2^s}
$$
where  for $N\in \N$ the sequences of functions  $\d_{1,N},\ldots,
\d_{\ell,N}\colon \N^k \to L^\infty(\mu)$ are  bounded by $1$.
 %%and $E_N\in \mathcal{S}_{b_1,\ldots,b_l}$.

 For $i=1,\ldots, k$,  let
$\mathcal{A}_i=(b_{1,i},\ldots, b_{\ell,i})$,
  and define
the matrix type $W$ of the family of $k$-tuples
$(\mathcal{A}_1,\ldots,\mathcal{A}_k)$  as in Section~\ref{SS:4.1}.
Notice that  having fixed $d,k, \ell$, there is only a finite number
of possibilities for $W$. The proof of the claim is going to proceed
by induction on $W$. We remark that it suffices to show that the
constants $C$ and $s$ depend only on $W, k$, and $\ell$.
Furthermore, we can assume that $b_{1,1}$ is the function with the
largest growth rate.

\noindent {\bf Base case:} We assume that $d=-1$, in which case all
functions  $b_{i,j}(t)$ converge to $0$. Then for $i=1,\ldots,
\ell$, for all large enough $n\in \N$  the sequence $[\b_i]$ takes
values on some finite subset $F_i\subset \Z^k$ with $|F_i|\leq
2^{k}$. Without loss of generality we can assume that this happens
for every $n\in \N$.
%%$$
%%G(n)=\E_{\m \in \N^k} \ \prod_{i=1}^\ell d_i(\m+ [\a_i(n)] ) \cdot
%%{\bf 1}_S(n)
%%$$
%%then
%%$$
%%G(n)=\sum_{i=1}^t G(n)\cdot  {\bf 1}_{S_{i}}(n)
%%$$
%%Then
%%$$
%%G(n)=\sum_{i=1}^t c_i\cdot {\bf 1}_{S_{i}}(n)
%%$$
%%for some constants $c_i$ with $|c_i|\leq 1$.
%%Then we have to estimate the expression
%%$$
%%\limsup_{N\to\infty}\limsup_{M\to\infty}\Big|\frac{1}{N} \sum_{n\in
%%[1,N]} \ \big( A(n)\cdot \sum_{j=1}^t
%%  \frac{1}{M^k}\sum_{\m \in [1,M]^k}  \
%%\prod_{i=1}^\ell \d_{i,N}(\m+ [\b_i(n)] ) \cdot {\bf
%%1}_{S_j}(n)\big)\Big|
%%$$
%%where
Let $E_{N,1}, \ldots, E_{N,t}$ ($t\leq 2^{k\ell}$) be sets that form
a partition of $S_N$ into sets where all the sequences $[b_{i,j}]$
are constant. Then there exist constants $|c_{j,N}|\leq 1$ such that
for $s=0$ the quantity we want to estimate is equal to
$$
\limsup_{N\to\infty}\norm{ \sum_{j=1}^t c_{j,N} \Big(
\frac{1}{N}\sum_{n=1}^{N} \ \big( A_x(n) \ \!  {\bf 1}_{E_{N,j}}(n)
 \big)\Big)}_{L^2(\mu)}\leq
t\ \! \limsup_{N\to\infty}\sup_{E\subset \N} \norm{
\frac{1}{N}\sum_{n=1}^N \big( A_x(n)\ \!
 {\bf 1}_{E }(n)
 \big)}_{L^2(\mu)}.
$$

\noindent {\bf Inductive step:} Let
$(\mathcal{A}_1,\ldots,\mathcal{A}_k)$ be a family of $\ell$ ordered
$k$-tuples of functions with matrix type $W$ and degree $d\geq 0$,
in which case $\deg(b_{1,1})\geq 0$. Suppose that the claim  holds
for every family of $2\ell$ ordered $k$-tuples of functions of
matrix type $W'$ strictly less than $W$ with $s_0=s_0(W',k,2\ell)$
and $C=C(W',k,2\ell)$. We let
\begin{equation}\label{E:sC}
s_0(W,k,\ell)=\max_{W'<W}(s_0(W',k,2\ell))+1, \quad
C(W,k,\ell)=2^{(2k\ell+1)2^{s_0(W,k,\ell)-1}}\max_{W'<W}(C(W',k,2\ell))
\end{equation}
where the max is taken over the finitely many matrix types of
families of at most $2\ell$ functions that are smaller than $W$. The
induction will be complete if we show that the asserted estimate
holds for the family $(\mathcal{A}_1,\ldots,\mathcal{A}_k)$ for
these values of $s(W,k,\ell)$ and  $C(W,k,\ell)$.

 We start by using the Cauchy Schwarz inequality
\begin{multline}\label{E:uty2}
 \limsup_{N\to\infty} \limsup_{M\to\infty}\norm{\frac{1}{N}\sum_{n=1}^N \ \big( A_x(n)\cdot
 \frac{1}{M^k}\sum_{\m \in [1,M]^k} \
\prod_{i=1}^\ell \d_{i,x,N}(\m+ [\b_i(n)] ) \cdot {\bf
1}_{E_N}(n)\big)}_{L^2(\mu)}^{2}\leq
\\   \limsup_{N\to\infty}\limsup_{M\to\infty} \frac{1}{M^k}\sum_{\m \in [1,M]^k}
\norm{\frac{1}{N}\sum_{n=1}^{N} \ \big( A_x(n)\cdot \prod_{i=1}^\ell
\d_{i,x,N}(\m+ [\b_i(n)] ) \cdot {\bf
1}_{E_N}(n)\big)}_{L^2(\mu)}^2.
\end{multline}
Using Lemma~\ref{L:VDC2}, ignoring negligible terms, and using the
Cauchy Schwarz inequality, we find
 that the last expression is bounded by   $2$ times
\begin{multline*}
\limsup_{H\to\infty}\frac{1}{H}\sum_{h=1}^{H} \limsup_{N\to\infty}
\limsup_{M\to\infty} \Big|\!\Big| \frac{1}{N}\sum_{n=1}^{N} \
\big(A_x(n+h)
\cdot \bar{A_x}(n) \cdot \\
 \frac{1}{M^k}\sum_{\m \in [1,M]^k} \
\prod_{i=1}^\ell \d_{i,x,N}(\m+ [\b_i(n+h)])\cdot
\bar{\d}_{i,x,N}(\m+ [\b_i(n)] ) \cdot {\bf 1}_{E_N}(n+h)\cdot {\bf
1}_{E_N}(n)\big)\Big|\!\Big|_{L^2(\mu)}.
\end{multline*}
  We make the change of
variables $\m\to \m-[\b(n)]$, for some vector valued function $\b$
that will be determined later.
%% and has   coordinates in the set
%%$\cF(b_1,\ldots, b_l)$.
Ignoring negligible terms, we see that the
last expression is equal to
\begin{multline}\label{E:qwa}
 \limsup_{H\to\infty} \frac{1}{H}\sum_{h=1}^{H}\limsup_{N\to\infty} \limsup_{M\to\infty}
  \Big|\!\Big|\frac{1}{N}\sum_{n=1}^{N} \
\big(A_x(n+h) \cdot \bar{A_x}(n) \cdot \\  \frac{1}{M^k}\sum_{\m \in
[1,M]^k} \ \prod_{i=1}^\ell \d_{i,x,N}(\m+
[\b_i(n+h)-\b(n)]+\e_{i,h}(n)) \cdot \bar{\d}_{i,x,N}(\m+
[\b_i(n)-\b(n)]+\e'_{i,h}(n) ) \cdot {\bf 1}_{E_{N,h}}(n)
\big)\Big|\!\Big|_{L^2(\mu)}
\end{multline}
where the sequences $(\e_{i,h}(n)), (\e'_{i,h}(n))$ take  values in
$\{0,1\}^k$ and $E_{N,h}\=E_N\cap (E_N-h)$.
%%\in
%%\mathcal{S}_{b_1,\ldots, b_l}$.
Notice that
\begin{multline*}
  \prod_{i=1}^\ell \d_{i,x,N}(\m+
[\b_i(n+h)-\b(n)]+\e_{i,h}(n))  \cdot \bar{\d}_{i,x,N}(\m+
[\b_i(n)-\b(n)]+\e'_{i,h}(n)) \cdot {\bf 1}_{E_{N,h}}(n)=\\
 \sum_{j=1}^t   \prod_{i=1}^\ell
\d_{i,j,x,N}(\m+ [\b_i(n+h)-\b(n)])  \cdot \d'_{i,j,x,N}(\m+
[\b_i(n)-\b(n)])\cdot {\bf 1}_{E_{N,h,j}}(n),
\end{multline*}
where the sets $E_{N,h,1}, \ldots, E_{N,h,t}$ ($t\leq 2^{2k\ell}$)
form a partition of $E_{N,h}$ into sets where the sequences $\e_i,
\e'_i$ are constant (either $0$ or $1$),
$\d_{i,j,N}(\m)\=\d_{i,N}(\m+c_j)$, and
$\d'_{i,jN}(\m)\=\bar{\d}_{i,N}(\m+c'_j)$. Combining the above we
get that the limit in \eqref{E:qwa} is bound by
\begin{multline*}
t\cdot  \limsup_{H\to\infty} \frac{1}{H}\sum_{h=1}^{H}
 \limsup_{N\to\infty} \sup_{\norm{d_i}_\infty, \norm{d'_i}_\infty\leq 1,
 E\subset \N} \limsup_{M \to\infty}
\Big|\!\Big|\frac{1}{N}\sum_{n=1}^{N} \ \big(A_x(n+h) \cdot
\bar{A_x}(n) \cdot
\\  \frac{1}{M^k}\sum_{\m \in [1,M]^k} \ \prod_{i=1}^\ell \d_{i,x}(\m+
[\b_i(n+h)-\b(n)]) \cdot \d'_{i,x}(\m+ [\b_i(n)-\b(n)] ) \cdot {\bf
1}_{E}(n) \big)\Big|\!\Big|_{L^2(\mu)}.
\end{multline*}

 This naturally leads us to consider a new family
that consist of $2\ell$ ordered $k$-tuples of functions. Choosing
$\b$ exactly as in the proof of Lemma~\ref{L:reduceA'}, and
following the argument used there,
%%\footnote{Instead of Lemma~\ref{L:properties} we  use that if
%%$a\in \H$ has degree $d\geq 0$, then $\deg(S_ha-a)=d-1$ for every
%%$h\in \N$.},
we see that this new family has matrix type $W'$
strictly smaller than $W$.

For this choice of $\b$,  raising both sides of \eqref{E:uty2} to
the power $2^{s(W',k,2\ell)}$,   working through the previous
estimates (we use also the Holder inequality at the last step), and
using the induction hypothesis, we get that for $s=s(W',k,2\ell)+1$
and $C=C(W,k,\ell)$, defined as in \eqref{E:sC},  the left hand side
in \eqref{E:uty1} is bounded by $C$ times
\begin{multline*}
\limsup_{H\to\infty}\frac{1}{H}\sum_{h=1}^{H}
\limsup_{H_s\to\infty}\frac{1}{H_s}\sum_{h_s=1}^{H_s} \cdots
\limsup_{H_1\to\infty}\frac{1}{H_1}\sum_{h_1=1}^{H_1}\limsup_{N\to\infty} \\
\sup_{E\subset \N}\Big|\!\Big| \frac{1}{N}\sum_{n=1}^{N}
\prod_{\e\in \{0,1\}^s} \cC^{|\e|}A_x(n+h+\e \cdot \h)\cdot
\cC^{|\e|}\bar A_x(n+\e \cdot \n) \cdot {\bf 1}_{E}(n)
\Big|\!\Big|_{L^2(\mu)}
\end{multline*}
where $\h=(h_1,\ldots, h_s)$. The last expression is equal to
$$
 \limsup_{H\to\infty}\frac{1}{H}\sum_{h=1}^{H} \frac{1}{H_s}\sum_{h_s=1}^{H_s}\! \cdots
\limsup_{H_1\to\infty}\frac{1}{H_1}\sum_{h_1=1}^{H_1}\limsup_{N\to\infty}
\sup_{E\subset \N}\Big|\!\Big| \frac{1}{N}\sum_{n=1}^{N} \!
\prod_{\e\in \{0,1\}^{s+1}}
 \cC^{|\e|} A_x(n+\e \cdot \h) \cdot {\bf 1}_{E}(n) \Big|\!\Big|_{L^2(\mu)}
$$
where $\h=(h_1,\ldots, h_s,h)$, as desired.
\end{proof}
\subsection{Correlation estimates for ergodic averages}
Next we combine Proposition~\ref{P:BASIC} with
Proposition~\ref{P:CharB} in order to prove a result that will be
crucial in the proof of Theorem~\ref{T:MainConv}.
\begin{proposition}\label{prop:weighted}
 Let $(X,\X,\mu,T_1,\ldots,T_\ell)$ be a system
and $f_1,\ldots,f_m\in L^\infty(\mu)$ be functions. Let
$(\A_1,\ldots,\A_\ell)$ be a nice  ordered  family of $\ell$-tuples
of functions with degree at most $d$ and such that
$\deg(a_{1,1})\geq 1$. Let $\H$ be a Hardy field and
  $b_1,\ldots, b_l\in \H$   be functions with maximum degree $d$.
Furthermore, for $i=1,\ldots, l$, let $(\cD_{i}(n)) $ be a sequence
of functions in $L^\infty(\mu)$, all bounded by $1$, such that for
almost every $x\in X$, the sequences $(\cD_{i,x}(n))$
 are  dual sequences of level   at most   $r\in \N$. Then there  exists $k=k(d,l,\ell,m,r )\in \N$ such
that: If
    $\nnorm{f_1}_{k,T_1}=0$, then the averages
\begin{equation}\label{E:098}
\frac{1}{N}\sum_{n=1}^{N}
 \prod_{i=1}^m f_i(T_1^{[a_{1,i}(n)]}  \cdots
 T_\ell^{[a_{\ell,i}(n)]}x)
\cdot \prod_{i=1}^l \cD_{i,x}([b_i(n)])
\end{equation}
 converge to $0$ in  $L^2(\mu)$.
\end{proposition}
\begin{proof}
Let $s\=s(d,l,r)$ be as in the statement of
Proposition~\ref{P:BASIC}.
 We assume that $\nnorm{f_1}_{k,T_1}=0$ where $k\=k(d,\ell,2^s \ell )$
 %%(i.e. $m=2^s \ell$)
  is given by Proposition~\ref{P:CharB}.
%%In order to prove the announced convergence to $0$, it suffices to
%%show that every increasing sequence $(N_k)$ admits a subsequence
%%$(N_k')$ such that
%%\begin{equation}
%%\label{eq:Inprime} \frac{1}{N_k'}\sum_{n=1}^{N_k'}
%% \prod_{i=1}^m f_i(T_1^{[a_{1,i}(n)]}\cdots
%% T_\ell^{[a_{\ell,i}(n)]}x)
%%\cdot \prod_{i=1}^l \cD_{i,[b_i(n)]}(x)\  \text{ converges to $0$ in
%%$L^2(\mu)$.}
%%\end{equation}
 We let $s'\=2^{s}$, and for $x\in X$, let $ (A_x(n)) $
be the sequence  of $L^\infty(\mu)$ functions defined by
$$
A_x(n)\=\prod_{i=1}^{m} f_i(T_1^{[a_{1,i}(n)]}\cdots
 T_\ell^{[a_{\ell,i}(n)]}x).
$$
%%For $r_1,\ldots,r_{s'}\in\N$, we study the $L^2(\mu)$ behavior of
%%the averages
%%$$
%%\frac 1{N}\sum_{n=1}^{N} \prod_{i=1}^{s'} A_x(n+r_i).
%%$$
For $i=1,\ldots, \ell$, consider the following ordered families
each consisting of $m s'$ functions:
$$
\A_i' \= \bigl(a_{i,1}(n+r_1),\ldots,a_{i,1}(n+r_{s'}), \ldots,
a_{i,m}(n+r_1),\ldots,a_{i,m}(n+r_{s'}) \bigr).
$$
Since  $\deg(a_{1,1})\geq 1$ and $(\A_1,\ldots,\A_\ell)$ is a nice
ordered family, one can  check  using Lemma~\ref{L:KeyHardy} that
$(\A_1',\ldots,\A_\ell')$ is also a  nice ordered family for all
${\bf r}$ in a subset  $R\subset \N^{s'}$ of the form
$$
R\=\{{\bf r}=(r_1,\ldots, r_{s'}) \colon   r_1\geq c_1,\ r_2\geq
c_{2}(r_1),\ldots, \ r_{s'}\geq c_{s'}(r_1,\ldots,r_{s'-1})\}
$$
for some sequences $c_i\colon \N^{i-1}\to \N$.
 Using Proposition~\ref{P:CharB}
we have that
%%If $f_1\bot \cZ_{k,T_1}$, then the averages
$$
\lim_{ N\to\infty} \sup_{E\subset \N} \norm{ \frac
1{N}\sum_{n=1}^{N} \prod_{i=1}^{s'} A_x(n+r_i) \cdot {\bf
1}_E(n)}_{L^2(\mu)}=0
$$
for all ${\bf r}\in R$.
%%By passing to  a subsequence $(N_k')$ of
%%$(N_k)$ we can assume that for every countable collection $(E_k)$ of
%%subsets of $\N$ we have
%%$$
%%\lim_{ k\to\infty}  \frac 1{N_k}\sum_{n=1}^{N_k} \prod_{i=1}^{s'}
%%A_{n+r_i}(x) \cdot {\bf 1}_{E_k}(n)=0
%%$$
%%pointwise almost everywhere  for a nice set of ${\bf r}\in \N^{s'}$.
Furthermore, a similar conclusion holds if one replaces some of the
sequences of functions $(A(n+r_i))_{n\in\N}$ with their complex
conjugates.
%%In fact, since the choice of index $r_1$ place no special role,
%%we get the same convergence unless

Hence,  for a set of $\h\in \N^{s}$ that has similar structure as
$R$, we have
$$
\lim_{N\to\infty} \sup_{E\subset \N }\norm{\frac{1}{N}\sum_{n=1}^{N}
\prod_{\e\in \{0,1\}^s} \cC^{|\e|} A_x(n+\e\cdot \h) \cdot {\bf
1}_{E}(n) }_{L^2(\mu)}=0.
$$
We deduce from   Proposition ~\ref{P:BASIC}  that the averages
\eqref{E:098} converge to $0$ in $L^2(\mu)$, as desired.
\end{proof}

\section{Seminorm estimates for the lower degree iterates and
proof of convergence}\label{S:Charlower} In this section we prove
Theorem~\ref{T:MainConv}. We first handle the case where all
  the iterates have super-linear growth, and later
 on
use an averaging trick to handle the general case.
\subsection{Seminorm estimates in the positive degree case}\label{SS:ProofA}
\begin{proposition}
\label{pr:charcnilseq'}
%%Let $\ell,d,s$ be positive integers.
Let $(X,\X,\mu,T_1,\ldots,T_\ell)$ be a system and
$f_1,\ldots,f_\ell\in L^\infty(\mu)$ be functions.  Let $\H$ be a
Hardy field and $a_1,\ldots,a_\ell \in \G\cap \H$ be functions with
different growth rates and degree between $1$ and
 $d$ for some $d\in\N$.
 Then there exists
$k=k(d,\ell)$ such that the following holds: If
$\nnorm{f_i}_{k,T_i}=0$ for some $i\in\{1,\ldots,\ell\}$, then the
averages
$$
 \frac{1}{N}\sum_{n=1}^{N}
 \prod_{i=1}^\ell T_i^{[a_i(n)]}f_i
$$
converge to $0$ in $L^2(\mu)$.
\end{proposition}
Proposition~\ref{pr:charcnilseq'} follows from the following more
general result:
\begin{proposition}
\label{pr:charcnilseq}
%%Let $\ell,d,s$ be positive integers.
Let $(X,\X,\mu,T_1,\ldots,T_\ell)$ be a system and
$f_1,\ldots,f_\ell\in L^\infty(\mu)$ be functions.  Let $\H$ be a
Hardy field and $a_1,\ldots,a_\ell \in \G\cap \H$  be functions with
different growth rates and degree between $1$ and
 $d$ for some $d\in\N$. Furthermore, let  $b_1,\ldots, b_l\in \H$
  have degree at most $d$.
For $i=1,\ldots, l$, let $(\cD_{i,x}(n))_{n\in\N}$ be  a uniformly
bounded sequence of measurable functions  such that, for almost
every $x\in X$, the sequence $(\cD_{i,x}(n))_{n\in\N}$ is a dual
sequence of level at most $r$.
 Then there exists
$k=k(d,l,\ell,r)$ such that the following holds: If
$\nnorm{f_i}_{k,T_i}=0$ for some $i\in\{1,\ldots,\ell\}$, then the
averages
\begin{equation}
\label{eq:weighted}
 \frac{1}{N}\sum_{n=1}^{N}
 \prod_{i=1}^\ell f_i(T_i^{[a_i(n)]}x) \cdot \prod_{i=1}^l
 \cD_{i,x}([b_i(n)])
\end{equation}
converge to $0$ in $L^2(\mu)$.
\end{proposition}
\begin{proof}
The proof goes by induction on the number $\ell$ of transformations.
For $\ell=1$, the result follows from  the case $\ell=1$ of
Proposition~\ref{prop:weighted}. We take $\ell\geq 2$, assume that
the results holds for $\ell-1$ transformations, and we are going to
prove that it holds for $\ell$ transformations.

Without loss of generality we can assume that $a_1$  is the fastest
growing function, and that all functions and dual sequences are
bounded by $1$. By Proposition~\ref{prop:weighted}, there exists
$k_0=k_0(d,l,\ell,r)$ such that, if $\nnorm{f_1}_{k_0,T_1}=0$, then
the averages~\eqref{eq:weighted} converge to $0$ in $L^1(\mu)$. Let
$k_1\=k(\tilde{d},\tilde{l}, \ell-1, \tilde{r})$ be the integer that
the induction hypothesis gives for  $\tilde{d}\=\max\{d,k_0\}$,
$\tilde{r}\=\max\{r,k_0\}$, and $\tilde{l}\=l+1$. Suppose that
$\nnorm{f_i}_{k_1,T_i}=0$ for some $i\in \{2,\ldots, \ell\}$. The
induction will be complete if we show that  the averages
\eqref{eq:weighted} converge to $0$ in $L^2(\mu)$.

Let $\varepsilon>0$. By Proposition~\ref{P:ApprDual}  we can express
$f_1$ as  $f_1=f_{s}+f_{u}+f_{e}$, where  $f_s, f_u, f_e\in
L^\infty(\mu)$, $\nnorm{f_{u}}_{k_0,T_1}=0$,
$\norm{f_{e}}_{L^1(\mu)} \leq \varepsilon$, and $f_s=\sum_{i=1}^m
c_i f_{s,i}$, for some  $m\in \N$, $c_i\in \R$,   $f_{s,i}\in
L^\infty(\mu)$,
 and for  almost
every $x\in X$ the sequences $(f_{s,i} (T^nx))_{n\in\N}$ are    dual
sequences of level at most $k_0$.
 As we
explained before, when computing the limit in $L^1(\mu)$ of the
averages \eqref{eq:weighted}, the contribution of the term $f_{u}$
is
 negligible. Furthermore, by the induction hypothesis,
the same holds for the contribution of the term $f_{s,i}$, for
$i=1,\ldots,m$,  and as a consequence for the term $f_s$. It remains
to  handle  the  contribution of the term $f_{e}$. When $f_1$ is
replaced by $f_e$, the $L^1(\mu)$ norm of the averages
\eqref{eq:weighted} can be  bounded by
$$
\norm{\frac{1}{N}\sum_{n=1}^{N}T_1^{[a_1(n)]}|f_e|}_{L^1(\mu)}\leq
\frac{1}{N}\sum_{n=1}^{N}\norm{T_1^{[a_1(n)]}|f_e|}_{L^1(\mu)}=
\norm{f_e}_{L^1(\mu)}\leq \varepsilon.
$$
 Since  $\varepsilon$ was arbitrary, we deduce that
the averages~\eqref{eq:weighted} converge to $0$ in $L^1(\mu)$, and
as a consequence in $L^2(\mu)$ (since all functions $f_i$ are
bounded). This completes    the proof.
\end{proof}
We also record a variant of this result that will be used later.
\begin{proposition}\label{T:Rn+r'}
Let $(X,\X,\mu,T_1,\ldots,T_\ell)$ be a system and
$f_1,\ldots,f_\ell\in L^\infty(\mu)$ be functions.  Let $\H$ be a
Hardy field and $a_1,\ldots,a_\ell \in \G\cap \H$  be functions with
different growth rates and degree between $1$ and
 $d$.
%%and satisfy $t^{k+\varepsilon} \prec a_i(t)\prec t^{k+1}$ for some $k=k_i\in \N$ and $\varepsilon>0$ .
Then there exists $k=k(d,\ell)$ such that the following holds: If
 $\nnorm{f_i}_{k,T_i}=0$, for some $i\in\{1,\ldots,\ell\}$, then
  \begin{equation}\label{E:product}
\lim_{R\to\infty}\limsup_{N\to\infty}\frac{1}{N}\sum_{n=1}^N\norm{
\frac{1}{R}\sum_{r=1}^R \prod_{i=1}^\ell T_i^{[a_i(Rn+r)]}f_i
}_{L^2(\mu)}=0.
  \end{equation}
\end{proposition}
\begin{proof}
Suppose that  $a_1$  is the fastest growing function and all
functions are bounded by $1$.  Notice  that for every $R\in \N$ we
have
 $$ \frac{1}{N}\sum_{n=1}^N\norm{ \frac{1}{R}\sum_{r=1}^R
\prod_{i=1}^\ell T_i^{[a_i(Rn+r)]}f_i}^2_{L^2(\mu)} =
\frac{1}{R^2}\sum_{1\leq r_1,r_2\leq R }  \frac{1}{N}\sum_{n=1}^N
\int \prod_{i=1}^\ell T_i^{[a_i(Rn+r_1)]}f_i \cdot
T_i^{[a_i(Rn+r_2)]}\bar{f}_i \ d\mu.
$$
For $r_1\neq r_2$,  using  Proposition~\ref{prop:weighted} (the
corresponding family of $\ell$-tuples is nice) we get that there
exists $k_0=k_0(d,\ell)$ such that if $\nnorm{f_1}_{k_0,T_1}=0$,
then   the averages
\begin{equation}\label{E:rn+r_1}
\frac{1}{N}\sum_{n=1}^N  \prod_{i=1}^\ell T_i^{[a_i(Rn+r_1)]}f_i
\cdot T_i^{[a_i(Rn+r_2)]}\bar{f}_i
\end{equation}
converge to $0$ in $L^2(\mu)$. It is then
 straightforward to adapt the proof of
Proposition~\ref{pr:charcnilseq} in order to get that there exists
$k=k(d,\ell)$ such that
  for $r_1\neq r_2$ , if $\nnorm{f_i}_{k,T_i}=0$, then  the
averages \eqref{E:rn+r_1} converge to $0$ in $L^2(\mu)$. We deduce
that for every $R\in \N$ we have
$$
\limsup_{N\to\infty}\frac{1}{N}\sum_{n=1}^N\norm{\frac{1}{R}\sum_{r=1}^R
\prod_{i=1}^\ell T_i^{[a_i(Rn+r)]}f_i}^2_{L^2(\mu)}\leq 1/R.
    $$
    Taking $R\to \infty$
    we deduce that \eqref{E:product} holds and completes the proof.
\end{proof}

\subsection{Equidistribution on nilmanifolds}
%%We are going to use  the following equidistribution result:
\begin{proposition}[\cite{Fr09}]\label{P:NilEqui}
 Let $\H$ be a Hardy field  and  $a_1,\ldots,a_\ell\in \G\cap \H$
 be functions with
different growth rates and positive degree. For $i=1,\ldots, \ell$,
let $X_i\=G_i/\Gamma_i$ be nilmanifolds, $b_i \in G_i$, and $x_i \in
X_i$. Then  the sequence
$$
(b_1^{[a_1(n)]}x_1,\ldots, b_\ell^{[a_\ell(n)]}x_\ell)
$$
is equidistributed on the nilmanifold    $\prod_{i=1}^\ell
\overline{\{b_i^nx_i\colon n\in\N\}}$.
\end{proposition}
For future use we record an identity that follows from the previous
result: For all functions $F_i\in C(X_i)$ we have
\begin{equation}\label{poi}
\lim_{N\to\infty} \frac{1}{N}\sum_{n=1}^N \prod_{i=1}^\ell
F_i(b_i^{[a_i(n)]}x_i)=\prod_{i=1}^\ell \lim_{N\to\infty}
\frac{1}{N}\sum_{n=1}^N
 F_i(b_i^{n}x_i).
\end{equation}
We are also going to use another identity. Its proof is essentially
contained in \cite{Fr09}.
 %%To avoid unnecessary repetition we only sketch
%%the proof.
\begin{proposition}\label{P:Rn+r}
 Let $\H$ be a Hardy field and   $a_1,\ldots,a_\ell\in \G\cap \H$   be functions with different growth rates
and positive degree. For $i=1,\ldots, \ell$, let $X_i\=G_i/\Gamma_i$
be nilmanifolds, $b_i\in G_i$, $x_i\in X_i$, and $F\in C(X)$, where
$X=X_1\times\cdots\times X_\ell$. Then \begin{equation}
\label{E:pwe}
\lim_{R\to\infty}\limsup_{N\to\infty}\frac{1}{N}\sum_{n=1}^N \Big|
\frac{1}{R}\sum_{r=1}^R
F(b_1^{[a_1(Rn+r)]}x_1,\ldots,b_\ell^{[a_\ell(Rn+r)]}x_\ell)-\int F
\ dm_{\tilde{X}}\Big|=0
\end{equation}
where   $\tilde{X}=\prod_{i=1}^\ell\overline{\{b_i^nx_i\colon
n\in\N\}}$.
\end{proposition}
\begin{proof}[Sketch of Proof]
Using a straightforward modification of the reduction argument  of
Section 5.2 in \cite{Fr09}, we can reduce matters to proving the
following statement:  ``For $i=1,\ldots, \ell$, let
$X_i=G_i/\Gamma_i$ be nilmanifolds, with $G_i$ connected and simply
connected, $x_i\in X_i$,  $b_i\in G_i$ act ergodically on $X_i$
(meaning the sequence $(b_i^nx_i)$ is equidistributed in $X_i$ for
every $x_i\in X_i$), and $F\in C(X)$, where $X=X_1\times\cdots\times
X_\ell$. Then \eqref{E:pwe} holds with $X$ in place of
$\tilde{X}$.''
%%$$
%%\lim_{R\to\infty}\limsup_{N\to\infty}\frac{1}{N}\sum_{n=1}^N \Big|
%%\frac{1}{R}\sum_{r=1}^R
%%F(b_1^{a_1(Rn+r)}x_1,\ldots,b_\ell^{a_\ell(Rn+r)}x_\ell)-\int F \
%%dm_{X}\Big|=0.
%%$$

This was verified while proving Proposition~5.3  in \cite{Fr09},
completing the proof.
\end{proof}
\subsection{Proof of Theorem~\ref{T:MainConv} in the positive degree case}

\begin{proposition}\label{TT:Rn+r}
Theorem~\ref{T:MainConv} holds when  all functions
$a_1,\ldots,a_\ell$ have positive degree.
\end{proposition}
\begin{proof}
We want to show that for every   system
$(X,\mathcal{B},\mu,T_1,\ldots, T_\ell)$
 and
   functions $f_1,\dots,f_\ell\in L^\infty(\mu)$, we have
  \begin{equation}\label{E:product'}
\lim_{N\to\infty}\frac{1}{N} \sum_{n=1}^N \prod_{i=1}^\ell
T_i^{[a_i(n)]}f_i=\prod_{i=1}^\ell  \tilde{f}_i
  \end{equation}
where   converge is taken in $L^2(\mu)$ and
$\tilde{f}_i\=\E(f_i|\cI_{T_i})$.
  By
Proposition~\ref{pr:charcnilseq}  there exists $k$ such that, if
$\nnorm{f_i}_{k,T_i}=0$ for some $i\in\{1,\ldots,\ell\}$, then the
limit of the averages in \eqref{E:product'} is $0$ where convergence
takes place in $L^2(\mu)$ (and hence in $L^1(\mu)$ as well).

Let $\varepsilon>0$. By Proposition~\ref{P:ApprDual}, for
$i=1,\ldots, \ell$ we can write $f_i=f_{i,s}+f_{i,u}+f_{i,e}$, where
$f_{i,s}, f_{i,u}, f_{i,e}\in L^\infty(\mu)$,
$\nnorm{f_{i,u}}_{k,T_i}=0$, $\norm{f_{i,e}}_{L^1(\mu)} \leq
\varepsilon$,   and $f_{i,s}\in L^\infty(\mu)$ are such that  for
almost every $x\in X$ the sequence $(f_{i,s} (T_i^nx))$ is a
$k$-step nilsequence, say $(\cN_{i,x}(n))$.
 As we
explained before, when computing the limit in $L^1(\mu)$  of the
averages  in \eqref{E:product'}, the contribution of the terms
$f_{i,u}$ is
 negligible.
Furthermore, the same holds for the  contribution of the terms
$f_{i,e}$. This follows  since for every $N\in\N$ the $L^1(\mu)$
norm of the averages in \eqref{E:product'} is bounded by a constant
multiple of
$$
\min_{i=1,\ldots,\ell}
\frac{1}{N}\sum_{n=1}^{N}\norm{T_i^{[a_i(n)]}|f_i|}_{L^1(\mu)}=
\min_{i=1,\ldots, \ell}\norm{f_i}_{L^1(\mu)}.
$$
Therefore, it remains to examine the contribution of the terms
$f_{i,s}$. In this case, the average in \eqref{E:product'} takes the
form
$$
\frac{1}{N}\sum_{n=1}^N \prod_{i=1}^\ell \cN_{i,x}([a_i(n)]).
$$
Using  identity \eqref{poi} %%recorded after the proof of
%%Proposition~\ref{P:NilEqui},
we get that the limit of this average
is
%%$$
%%\prod_{i=1}^\ell \lim_{N\to\infty}\frac{1}{N}\sum_{n=1}^N \
%%\cN_{i,x}(n )
%%$$
$$
\prod_{i=1}^\ell \lim_{N\to\infty}\frac{1}{N}\sum_{n=1}^N \
\cN_{i,x}(n )
$$
which in turn is equal to
$$
\prod_{i=1}^\ell  \lim_{N\to\infty}\frac{1}{N}\sum_{n=1}^N \
f_{i,s}(T_i^nx).
$$
For reasons explained before this is equal, up to a constant
multiple of $\varepsilon$, to
$$
\prod_{i=1}^\ell \lim_{N\to\infty}\frac{1}{N}\sum_{n=1}^N \
f_i(T_i^nx) =\prod_{i=1}^\ell  \tilde{f}_i.
$$
Letting $\varepsilon\to 0$ completes the proof.
\end{proof}
The  proof of the next result is completely analogous to the proof
of Proposition~\ref{TT:Rn+r}, one uses Proposition~\ref{T:Rn+r'} in
place of Proposition~\ref{pr:charcnilseq} and
Proposition~\ref{P:Rn+r} in place of Proposition~\ref{P:NilEqui}
\begin{proposition}\label{TT:Rn+r'}
Let $(X,\X,\mu,T_1,\ldots,T_\ell)$ be a system and
$f_1,\ldots,f_\ell\in L^\infty(\mu)$ be functions.  Let $\H$ be a
Hardy field and $a_1,\ldots,a_\ell \in\G\cap \H$ be functions with
different growth rates and positive degree.
%%and satisfy $t^{k+\varepsilon} \prec a_i(t)\prec t^{k+1}$ for some $k=k_i\in \N$ and $\varepsilon>0$ .
Then
 %% \begin{equation}\label{E:product''}
$$
\lim_{R\to\infty}\limsup_{N\to\infty}\frac{1}{N}\sum_{n=1}^N \norm{
\frac{1}{R}\sum_{r=1}^R \prod_{i=1}^\ell  T_i^{[a_i(Rn+r)]}f_i -
\prod_{i=1}^\ell \tilde{f}_i}_{L^2(\mu)}=0
$$
%%  \end{equation}
where $\tilde{f}_i=\E(f_i|\cI_{T_i})$.
\end{proposition}

\subsection{Proof of Theorem~\ref{T:MainConv} in the general case}\label{SS:proof}

\begin{proof}[Proof of Main Theorem in the general case]
Without loss of generality we can assume that $a_\ell\prec
a_{\ell-1}\prec\cdots\prec a_1$. If all   functions
$a_1,\ldots,a_\ell$ have degree $0$, then the result follows from
Theorem~2.7 in \cite{Fr11}. If all   functions $a_1,\ldots,a_\ell$
have positive degree, then  the result was proved in the previous
subsection. Hence, we can assume that  there exists $m\in
\{1,\ldots, \ell-1\}$ such that $\deg(a_i)=0$ for $i=m+1,\ldots,
\ell$ and $\deg(a_i)\geq 1$ for $i=1,\ldots, m$.

It suffices to show that   if $\tilde{f}_i=0$ for some $i\in
\{1,\ldots,m\}$, then
\begin{equation}\label{EE:main}
\limsup_{N\to\infty} \norm{\frac{1}{N}\sum_{n=1}^N \prod_{i=1}^\ell
 T_i^{[a_i(n)]}f_i}_{L^2(\mu)} =0
\end{equation}
where the convergence  takes place in $L^2(\mu)$.   For every
$R\in\N$ the limit in \eqref{EE:main} is equal to
\begin{equation}\label{EE:main1}
  \limsup_{N\to\infty}
 \norm{\frac{1}{N}\sum_{n=1}^N
 \frac{1}{R}\sum_{r=1}^R
 \prod_{i=1}^\ell  T_i^{[a_i(nR+r)]}f_i}_{L^2(\mu)} .
\end{equation}
Since  the functions $a_{m+1},\ldots, a_\ell\in \H$ have  degree
$0$, it is easy to see  the following (one uses that their
derivative converges to $0$ and the mean value theorem):  for every
$R\in \N$, for a set of $n\in \N$ of density $1$, we have
$[a_i(nR+r)]=[a_i(nR)]$ for $r=1,\ldots,R$ and $i=m+1,\ldots,\ell$.
We deduce that  the limit in \eqref{EE:main1} is equal to
$$
\limsup_{N\to\infty} \norm{ \frac{1}{N}\sum_{n=1}^N \Big(
\prod_{i=m+1}^\ell
 T_i^{[a_i(nR)]}f_i \cdot
\frac{1}{R}\sum_{r=1}^R \prod_{i=1}^m T_i^{[a_i(nR+r)]}f_i
\Big)}_{L^2(\mu)}.
$$
This  is bounded by a constant times
$$
\limsup_{N\to\infty} \frac{1}{N}\sum_{n=1}^N \norm{
 \frac{1}{R}\sum_{r=1}^R
\prod_{i=1}^m  T_i^{[a_i(nR+r)]}f_i }_{L^2(\mu)}.
$$
 Using Proposition~\ref{TT:Rn+r'}
we see that the limit of this expression as $R\to\infty$ is equal to
$0$. This completes the proof.
\end{proof}


\begin{thebibliography}{99}

\bibitem{As10}
I.~Assani. Pointwise convergence of ergodic averages along cubes.
{\em J. Analyse Math.} {\bf 110} (2010), 241--269.

%% \bibitem{Au09} T.~Austin. On the norm convergence of nonconventional ergodic averages.
%%\emph{Ergodic Theory Dynam. Systems} {\bf 30} (2010), 321--338.



%%\bibitem{Au11a}
%%T.~Austin. Pleasant extensions retaining algebraic structure, I.
%% Preprint,  \texttt{arXiv:0905.0518v4}.

%%\bibitem{Au11b}
%%T.~Austin. Pleasant extensions retaining algebraic structure, II.
%% Preprint,  \texttt{arXiv:0910.0907v3}.


\bibitem{Be87a} V.~Bergelson. Weakly mixing PET. {\em Ergodic Theory
Dynam. Systems} \textbf{7} (1987), no. 3, 337--349.

\bibitem{Be06a} V.~Bergelson.  Combinatorial and Diophantine Applications of Ergodic Theory
(with appendices by A. Leibman and by A. Quas and M. Wierdl). {\em  Handbook of Dynamical Systems},
Vol. 1B, B. Hasselblatt and A. Katok, eds., Elsevier, (2006),  745--841.

\bibitem{Be06b} V.~Bergelson. Ergodic Ramsey Theory: a dynamical approach to static theorems.
 {\em Proceedings of the International Congress of Mathematicians}, Madrid 2006,  Vol. II, 1655--1678.


\bibitem{BH09} V.~Bergelson, I.~H{\aa}land-Knutson. Weak mixing implies mixing of higher orders along tempered functions. {\em Ergodic
    Theory Dynam. Systems} {\bf 29} (2009), no. 5,  1375--1416.


\bibitem{BHK05} V.~Bergelson, B.~Host, B.~Kra, with an appendix by I. Ruzsa.
Multiple recurrence and nilsequences. {\em Inventiones Math.} {\bf
160} (2005), no. 2, 261--303.

\bibitem{BL96}
V.~Bergelson, A.~Leibman. Polynomial extensions of van der Waerden's
and Szemer\'edi's theorems.  {\em J. Amer. Math. Soc.} \textbf{9}
(1996), 725--753.

\bibitem{BLL08} V.~Bergelson, A.~Leibman, E.~Lesigne.
  Intersective polynomials and the polynomial Szemer\'edi theorem. {\it Adv. Math.} \textbf{219}
  (2008), no. 1, 369--388.

\bibitem{Bo81} M.~Boshernitzan. An extension of Hardy's class L
  of ``Orders of Infinity".  {\em J. Analyse Math.} \textbf{39}
  (1981), 235--255.

%%\bibitem{Bos82} M.~Boshernitzan. New ``Orders of Infinity".  {\em
 %%   J. Analyse Math.} \textbf{41} (1982), 130--167.

    \bibitem{Bos94} M.~Boshernitzan. Uniform distribution and Hardy
  fields.  {\em J. Analyse Math.} \textbf{62} (1994), 225--240.

  \bibitem{BKQW05} M.~Boshernitzan, G.~Kolesnik, A.~Quas,
  M.~Wierdl.  Ergodic averaging sequences.  {\it J. Analyse Math.}
  \textbf{95} (2005), 63--103.

\bibitem{Bour61} N.~Bourbaki. Fonctions d'une variable r\'eele.
  Chapitre V (\'Etude Locale des Fonctions), 2nd edition, Hermann,
  Parris, 1961.

\bibitem{Chu11} Q.~Chu.
Multiple recurrence for two commuting transformations.  {\em Ergodic
Theory Dynam. Systems} \textbf{31} (2001), no.3, 771--792.


\bibitem{CFH11} Q.~Chu, N.~Franzikinakis, B.~Host. Ergodic averages of commuting
 transformations with distinct degree polynomial iterates.
 {\em Proc. Lond. Math. Soc.}    \textbf{102} (2011), 801--842.

%%\bibitem{CL88a} J-P.~Conze, E.~Lesigne.
%%Sur un th\'eor\`eme ergodique pour des mesures diagonales. {\em
%%Probabilit\'es, Publ. Inst. Rech. Math. Rennes}, 1987-1, Univ.
%%Rennes I, Rennes, (1988), 1--31.

%%\bibitem{Fr04} N.~Frantzikinakis.
%%The structure of strongly stationary systems. {\em J. Analyse Math.}
%%\textbf{93} (2004), 359--388.



%%\bibitem{Fr08} N.~Frantzikinakis.  Multiple ergodic averages for
 %% three polynomials and applications.   {\it Trans. Amer.
 %%   Math. Soc.}  \textbf{360} (2008), no. 10, 5435--5475.

\bibitem{Fr09} N.~Frantzikinakis. Equidistribution of sparse sequences on nilmanifolds.
 {\em J. Analyse Math.} {\bf 109} (2009), 1--43.


\bibitem{Fr10} N.~Frantzikinakis. Multiple recurrence and convergence for Hardy sequences of polynomial
growth.  {\em J. Analyse Math.}  \textbf{112} (2010), 79--135.

\bibitem{Fr11} N.~Frantzikinakis. Some open problems on multiple
ergodic averages. March 2011 \texttt{arXiv:1103.3808}.


%%\bibitem{FrK06} N.~Frantzikinakis,  B. Kra. Ergodic averages for
%%independent polynomials and applications. {\em J. London Math. Soc.}
%%\textbf{74} (2006), no. 1,  131--142.

%%\bibitem{FrLW09} N.~Frantzikinakis, E.~Lesigne, M.~Wierdl.
%%  Powers of sequences and recurrence.  {\it Proc. Lond. Math. Soc.} (3)
 %% \textbf{98} (2009), no. 2,  504--530.


\bibitem{FrLW11} N.~Frantzikinakis, E.~Lesigne, M.~Wierdl.
Random sequences and pointwise convergence of multiple ergodic
averages. To appear in Indiana Univ. Math. J.,
\texttt{arXiv:1012.1130}.

%%\bibitem{FrW09} N.~Frantzikinakis, M.~Wierdl. A Hardy field extension
%%of Szemer\'edi's theorem.  {\em Adv. Math.} \textbf{222}, (2009),
%%1--43.


\bibitem{Fu77} H.~Furstenberg.
Ergodic behavior of diagonal measures and a theorem of Szemer\'edi
on arithmetic progressions. {\em J. Analyse Math.} \textbf{31}
(1977), 204--256.

\bibitem{Fu81a} H.~Furstenberg.
Recurrence in ergodic theory and combinatorial number theory. {\em
Princeton University Press}, Princeton, 1981.



\bibitem{FuK79} H.~Furstenberg, Y.~Katznelson.
 An ergodic Szemer\'edi theorem for commuting transformations.
{\em J. Analyse Math.} \textbf{34} (1979), 275--291.

%%\bibitem{FuW96}  H.~Furstenberg, B.~Weiss.
 %%A mean ergodic theorem for $(1/N)\sum\sp N\sb {n=1}$
%%$f(T\sp nx)$ $\ g(T\sp {n\sp 2}x)$. Convergence in ergodic theory
%%and probability (Columbus, OH, 1993),
%% Ohio State Univ. Math. Res. Inst. Publ., {\bf 5}, de Gruyter, Berlin,
%%(1996), 193--227.

\bibitem{Gow01} W.~Gowers. A new proof of Szemer\'edi's theorem.
{\em Geom. Funct. Anal.} {\bf 11} (2001), 465--588.


\bibitem{GT09c} B.~Green, T.~Tao.  The quantitative behaviour of
  polynomial orbits on nilmanifolds.  To appear in {\it Annals Math.}, \texttt{arXiv:0709.3562}.


%%\bibitem{Gr11} J.~Griesmer.
%%Abundant configurations in sumsets with one dense summand.
%%\emph{Preprint}, \texttt{arXiv:1011.4657}.


\bibitem{Ha10} G.~Hardy.  Orders of Infinity. The ëInfinitarcalcuulí of Paul du Bois-Reymond.
 Reprint of the 1910 edition. {\em Cambridge
    Tracts in Math. and Math. Phys.}, 12,  Hafner Publishing Co., New York, 1971.

%%\bibitem{Ho09} B.~Host. Ergodic seminorms for commuting transformations and applications.
%%{\emph Studia Math.} {\bf 195} (1) (2009), 31--49.


\bibitem{HK05a} B.~Host,  B.~Kra.
Non-conventional ergodic averages and nilmanifolds. {\em Annals
Math.}  \textbf{161}  (2005), 397--488.


 \bibitem{HK05b} B.~Host, B.~Kra.  Convergence of polynomial
   ergodic averages. {\it Isr. J. Math.} \textbf{149} (2005), 1--19.



  \bibitem{HK09} B.~Host, B.~Kra. Uniformity seminorms on $l^{\infty}$ and applications.
{\em J. Analyse Math.} \textbf{108} (2009), 219--276.



%%\bibitem{Joh11} M.~Johnson.
%%Convergence of polynomial ergodic averages for some commuting
%%transformations.
%%%%  \emph{Illinois J. Math.} \textbf{53} (2009), no. 3, 865-882.



\bibitem{Kra06b} B.~Kra.
From combinatorics to ergodic theory and back again. {\em
Proceedings of International Congress of Mathematicians}, Madrid
2006, Vol. III, 57--76.


\bibitem{Kra11} B.~Kra. Poincar\'e recurrence and number theory: thirty years
later. {\em Bull. Amer. Math. Soc.} \textbf{48} (2011), 497--501.

\bibitem{KN74} L.~Kuipers, H.~Niederreiter.  Uniform distribution
  of sequences.  Pure and Applied Mathematics. {\em
    Wiley-Interscience}, New York-London-Sydney, 1974.

\bibitem{Lei05a} A.~Leibman.
Pointwise convergence of ergodic averages for polynomial sequences
of rotations of a nilmanifold. {\em Ergodic Theory Dynam. Systems}
{\bf 25}  (2005), no. 1,   201--213.


 \bibitem{Lei05c} A.~Leibman.  Convergence of multiple ergodic
   averages along polynomials of several variables.  {\it Isr. J.
     Math.} \textbf{146} (2005), 303--316.

%%\bibitem{Lei07} A.~Leibman. Orbits on a nilmanifold under the action of a polynomial sequence of translations. {\em Ergodic
%%heory Dynam. Systems} \textbf{27}  (2007), 1239--1252.

%%\bibitem{Lei09} A.~Leibman. Orbit of the diagonal in the power of a nilmanifold. {\em Trans. Amer. Math. Soc.}
%%{\bf 362} (2010), 1619--1658.


    \bibitem{Pav08} R.~Pavlov. Some counterexamples in topological
    dynamics.  {\em Ergodic Theory Dynam. Systems} {\bf 28} (2008), no. 4, 1291--1322.



\bibitem{Ta08} T.~Tao.
Norm convergence of multiple ergodic averages for commuting
transformations. {\em Ergodic Theory Dynam. Systems} \textbf{28}
(2008), no. 2,  657--688.

\bibitem{W11} M.~Walsh. Norm convergence of nilpotent ergodic
averages. Preprint, \texttt{arXiv:1109.2922}.

\bibitem{Zi07} T.~Ziegler.
Universal characteristic factors and Furstenberg averages. {\em J.
Amer. Math. Soc.} \textbf{20} (2007), 53--97.


%%\bibitem{Zi06} T.~Ziegler. Nilfactors of $\R^m$-actions and configurations
%%in sets of positive upper density in $\R^m$. {\em J. Analyse Math.}
%%\textbf{99}  (2006), 249--266.

%%\bibitem{WZ11} T.~Wooley, T.~Ziegler.  Multiple recurrence and
%%convergence along the primes. To appear in {\em Amer. J. of Math.},
%%\texttt{ arXiv:1001.4081}.

\end{thebibliography}
\end{document}